\documentclass[runningheads]{llncs}
\usepackage[T1]{fontenc}
\usepackage{graphicx}
\usepackage{amsmath}
\usepackage{amssymb}
\usepackage{enumerate}

\usepackage{algorithm}
\usepackage{algpseudocode}

\DeclareMathOperator{\CA}{CA}
\DeclareMathOperator{\CAN}{CAN}
\DeclareMathOperator{\CAS}{CA^*}
\DeclareMathOperator{\kmax}{k^{\max}}

\newcommand{\BFM}{{\bf M}}

\newcommand{\BFW}{{\bf W}}
\newcommand{\BFone}{{\bf 1}}
\usepackage{ mathrsfs }

\newcommand{\Mb}{\text{M}_{\text{B}}}
\newcommand{\Mzero}{\text{M}_{0}}
\newcommand{\Mone}{\text{M}_{1}}

\usepackage{fontawesome5}

\newtheorem{conj}{Conjecture}

\usepackage{hyperref}
\usepackage{xurl}

\usepackage{algorithm}
\usepackage{algpseudocode}
\usepackage{graphicx}
\usepackage{tikz}
\usepackage{pgfplots}
\usepackage{pgfplotstable}
\usepackage{xcolor}
\usepackage{subcaption}
\usepackage{placeins}
\usepackage{ifthen}
\pgfplotsset{compat=1.17}
\usepackage{stackengine}
\usepackage{graphicx}
\usepackage{graphbox,graphicx}
\usepackage{adjustbox}
\usepackage{pgffor}
\usepackage{booktabs}
\usepackage{multirow}
\usepackage[normalem]{ulem}
\usepackage{soul}
\sethlcolor{yellow}

\newcommand{\MzeroMILP}{
\begin{subequations}
\begin{align}
    \min ~ & \sum_{i \in \mathcal{P}} \sum_{j \in \mathcal{O}_i} \sum_{c \in \mathcal{C}^j \setminus \mathcal{B}} \rlap{$\left(\left| j - \lfloor N/2 \rfloor \right| + 1\right) V_{i c}$} \span \label{lin:objp3} \\
        \text{s.t.} ~ &
                          \sum_{c \in \mathcal{C}_i} V_{i c} = 1 & \forall i \in \mathcal{P} \label{lin:vpc1} \\
                        & \sum_{i \in \mathcal{P}_c} V_{i c} \leq 1 & \forall c \in \mathcal{C} \label{lin:vpc2} \\
                        & \sum_{j \in \mathcal{O}_i} W_{i j} = 1 & \forall i \in \mathcal{P} \label{lin:wpi1} \\
                        & V_{i c} \leq W_{i j} & \forall i \in \mathcal{P}, j \in \mathcal{O}_i, c \in \mathcal{C}^j \label{lin:wpi2} \\
                        & W_{i_1 j_1} + \sum_{j_2 \in \mathcal{O}_{i_2} : j_2 \leq N-j_1} W_{i_2 j_2} \leq 1 &  
                        \begin{aligned} \forall (i_1,i_2) \in \mathcal{F} : \\ \mathscr{F}(i_1,i_2) = (0,0) \end{aligned} \label{lin:unb1p3} \\
                        & W_{i_1 j_1} + \sum_{j_2 \in \mathcal{O}_{i_2} : j_2 \geq j_1} W_{i_2 j_2} \leq 1 &  
                        \begin{aligned} \forall (i_1,i_2) \in \mathcal{F} : \\ \mathscr{F}(i_1,i_2) = (0,1) \end{aligned} \\
                        & W_{i_1 j_1} + \sum_{j_2 \in \mathcal{O}_{i_2} : j_2 \leq j_1} W_{i_2 j_2} \leq 1 &  
                        \begin{aligned} \forall (i_1,i_2) \in \mathcal{F} : \\ \mathscr{F}(i_1,i_2) = (1,0) \end{aligned} \\
                        & W_{i_1 j_1} + \sum_{j_2 \in \mathcal{O}_{i_2} : j_2 \geq N-j_1} W_{i_2 j_2} \leq 1 &  
                        \begin{aligned} \forall (i_1,i_2) \in \mathcal{F} : \\ \mathscr{F}(i_1,i_2) = (1,1) \end{aligned} \label{lin:unb2p3} \\
                        & V_{i_1 c_1} + V_{i_2 c_2} \leq 1 & 
                        \begin{aligned} \forall (i_1, i_2) \in \mathcal{F}, \\ (c_1, c_2) \in \mathcal{I}^{\mathscr{F}(i_1,i_2)} \end{aligned} \label{lin:vv1} \\
                        & V_{i_1 c_1} + V_{i_2 c_2} \leq 1 & \forall (i_1, i_2) \in \overline{\mathcal{F}}, (c_1, c_2) \in \mathcal{I} \label{lin:vv2} \\
                        & V_{i c} \leq U_c &  \forall i \in \tilde{\mathcal{P}}, c \in \mathcal{C}_i \cap \mathcal{B} \label{lin:uuse} \\
                        & V_{i c} + U_{c'} \leq 1 & \begin{aligned} \forall i \in \mathcal{P}, \\ c \in \mathcal{C}_i, c' \in \mathcal{B}, (c,c') \in \mathcal{I} \end{aligned} \label{lin:uconflict} \\
                        & V_{i c} \in \{0, 1\} & \forall i \in \mathcal{P}, c \in \mathcal{C}_i \\
                        & W_{i j} \in \{0, 1\} & \forall i \in \mathcal{P}, j \in \mathcal{O}_i \\
                        & U_c \in \{0, 1\} & \forall c \in \mathcal{B} \label{lin:lastp3}
\end{align}
\end{subequations}
}

\newcommand{\MbaselineMILP}{
\begin{subequations}
\begin{align}
\min ~ & \sum_{r \in \mathcal{R}} Z_r \label{eq:big_obj} 
\\
\text{s.t.} ~ & (Y_{r i_1 i_2 \alpha \beta} = 1) \rightarrow (M_{k i_1} = \alpha) & \begin{array}{r} \forall r \in \mathcal{R}, (i_1, i_2) \in \overline{\mathcal{F}}, \\ (\alpha, \beta) \in \{0,1\} \times \{0, 1\} \end{array} \label{eq:big_ym1} 
            \\
            & (Y_{r i_1 i_2 \alpha \beta} = 1) \rightarrow (M_{k i_2} = \beta) & \begin{array}{r} \forall r \in \mathcal{R}, (i_1, i_2) \in \overline{\mathcal{F}}, \\ (\alpha, \beta) \in \{0,1\} \times \{0, 1\} \end{array} \label{eq:big_ym2} 
            \\
            & \sum_{r \in \mathcal{R}} Y_{r i_1 i_2 \alpha \beta} \geq 1 & \begin{array}{r} \forall (i_1, i_2) \in \overline{\mathcal{F}}, \\ (\alpha, \beta) \in \{0,1\} \times \{0, 1\} \end{array}  \label{eq:big_y11}
            \\
            & \sum_{r \in \mathcal{R}} Y_{r i_1 i_2 \alpha \beta} \geq 1 & \begin{array}{r} \forall (i_1, i_2) \in \mathcal{F}, \mathscr{F}(i_1,i_2) \neq (\alpha, \beta), \\  (\alpha, \beta) \in \{0,1\} \times \{0, 1\} \end{array} \label{eq:big_y12}
            \\
            & \sum_{r \in \mathcal{R}} Y_{r i_1 i_2 \alpha \beta} = 0 & \begin{array}{r} \forall (i_1, i_2) \in \mathcal{F}, \mathscr{F}(i_1,i_2) = (\alpha, \beta), \\  (\alpha, \beta) \in \{0,1\} \times \{0, 1\} \end{array} \label{eq:big_y0}
            \\
            & (M_{k i_1} = \alpha) \rightarrow (M_{k i_1} = 1 - \beta) &  \begin{array}{r} \forall (i_1, i_2) \in \mathcal{F}, \mathscr{F}(i_1,i_2) = (\alpha, \beta), \\  (\alpha, \beta) \in \{0,1\} \times \{0, 1\} \end{array} \label{eq:big_mm1}
            \\
            & (M_{k i_2} = \beta) \rightarrow (M_{k i_1} = 1 - \alpha) &  \begin{array}{r} \forall (i_1, i_2) \in \mathcal{F}, \mathscr{F}(i_1,i_2) = (\alpha, \beta), \\  (\alpha, \beta) \in \{0,1\} \times \{0, 1\} \end{array} \label{eq:big_mm2}
            \\
            & Y_{r i_1 i_2 \alpha \beta} \leq Z_r & \begin{array}{r} \forall r \in \mathcal{R}, (i_1, i_2) \in \overline{\mathcal{F}}, \\ (\alpha, \beta) \in \{0,1\} \times \{0, 1\} \end{array} \label{eq:big_yz}
            \\
            & M_{r i} \in \{0, 1\} & \forall r \in \mathcal{R}, i \in \mathcal{P} \\
            & Y_{r i_1 i_2 \alpha \beta} \in \{0, 1\} & \begin{array}{r} \forall r \in \mathcal{R}, (i_1, i_2) \in \overline{\mathcal{F}}, \\ (\alpha, \beta) \in \{0,1\} \times \{0, 1\} \end{array} \\
            & Z_r \in \{0, 1\} & \forall r \in \mathcal{R} \label{eq:big_last}
\end{align}
\end{subequations}
}

\newcommand{\FigPerformance}{
\begin{figure}[b!]
\centering
\includegraphics[width=\textwidth]{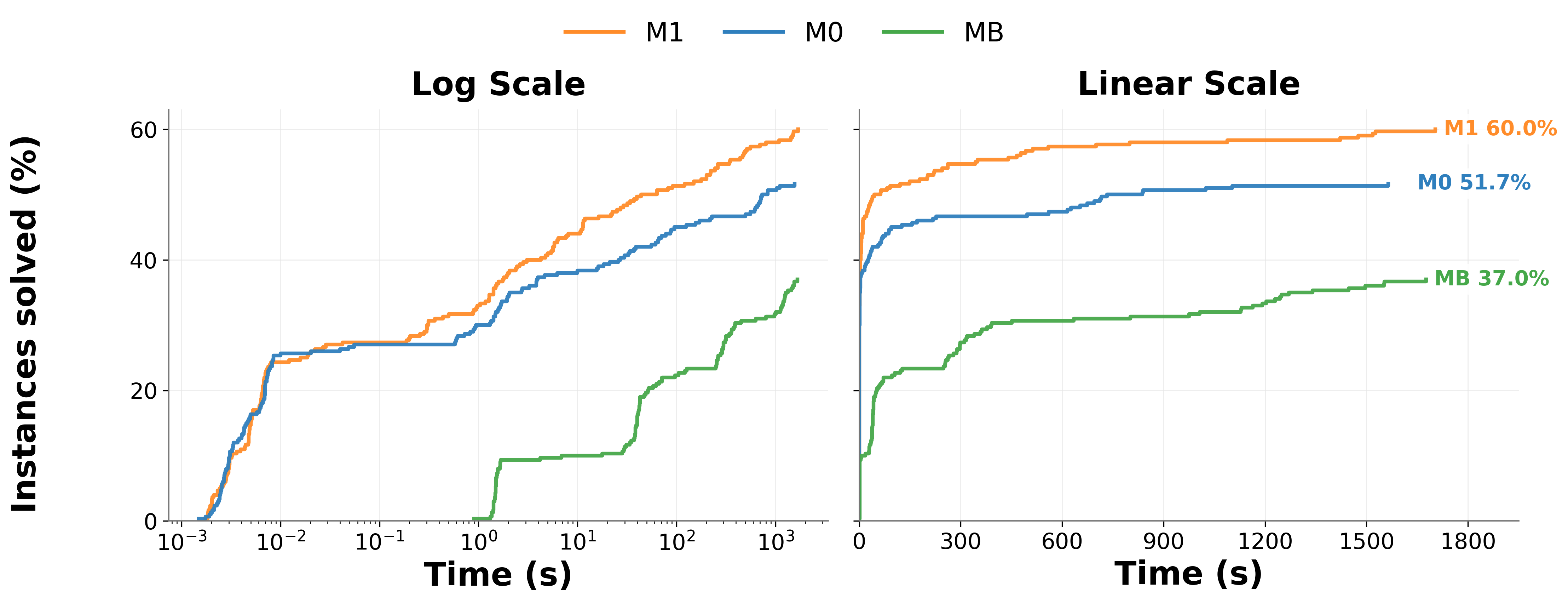}
\caption{Cumulative ratio of instances solved by each exact method over time, in logarithmic time scale (left) and linear time scale (right).}
\label{fig:performance}
\end{figure}
}

\newcommand{\FigExactRows}{
\begin{figure}[h!]
\centering
\includegraphics[width=\textwidth]{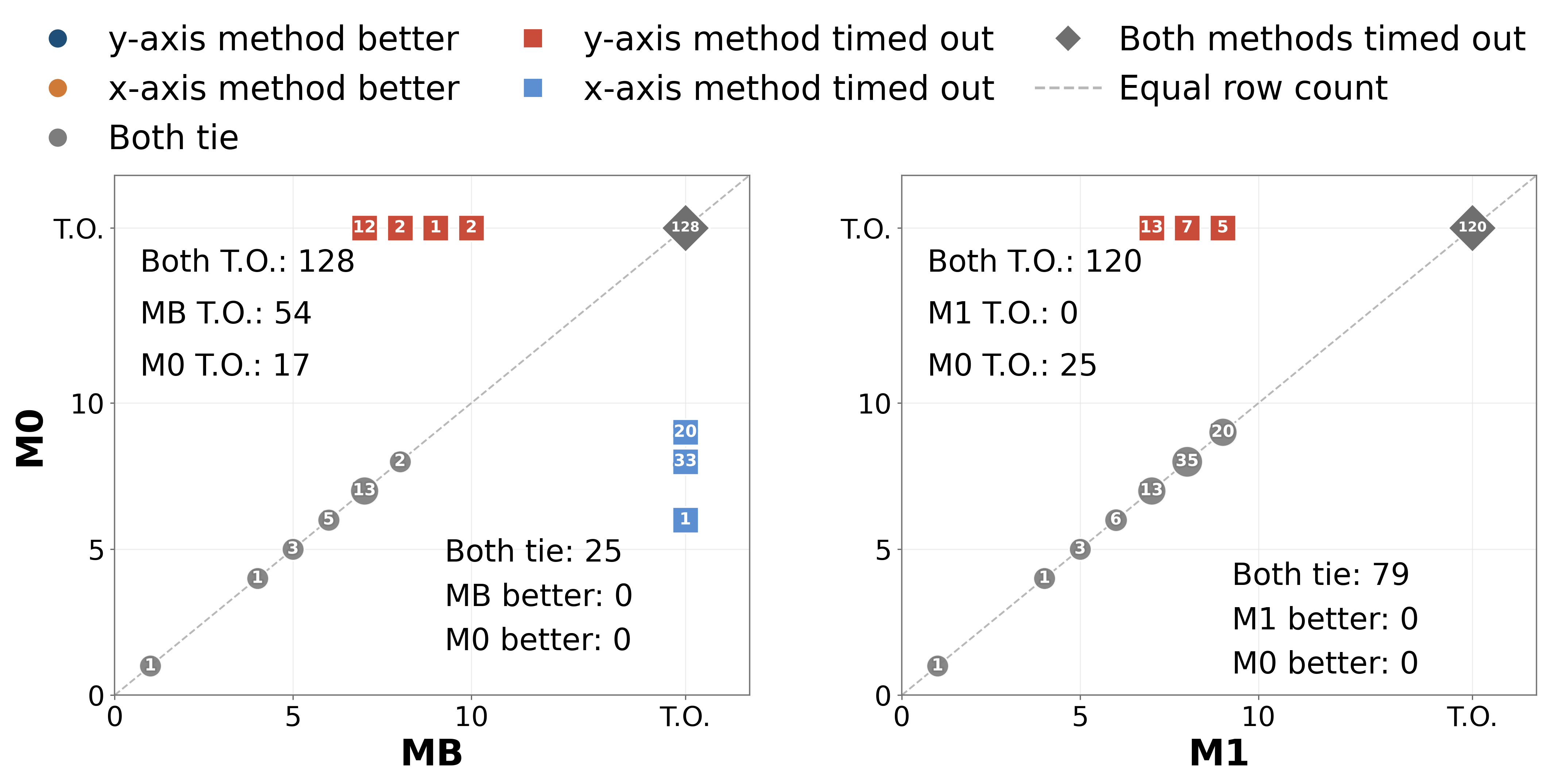}
\caption{Pairwise comparison of solution values (numbers of rows) and timeouts for exact methods, 
contrasting $\Mb$ with $\Mzero$ (left) and $\Mone$ with $\Mzero$ (right).}
\label{fig:exact-rows}
\end{figure}
}

\newif\ifpreprint
\preprinttrue

\begin{document}
\title{Optimal Combinatorial Testing with Constraints: The Balancing Act}
%
%\titlerunning{Abbreviated paper title}
% If the paper title is too long for the running head, you can set
% an abbreviated paper title here
%
\author{
Thiago Serra\inst{1} \and
Changkun Guan\inst{2} \and
Hunter Gehman\inst{3} \and
Sumit Dhar\inst{3} \and
\\ Mikey Ferguson\inst{4} \and
John Hooker\inst{5} \and
Marcel Schoppers 
}
%Third Author\inst{3}\orcidID{2222--3333-4444-5555}}
%\author{Anonymous Authors}
%
\authorrunning{T. Serra et al.}
% First names are abbreviated in the running head.
% If there are more than two authors, 'et al.' is used.
%
\institute{
University of Iowa, Iowa City IA, United States \\
\email{thiago-serra@uiowa.edu} \and
Georgia Institute of Technology, Atlanta GA, United States \and 
Bucknell University, Lewisburg PA, United States \and
University of Tulsa, Tulsa OK, United States \and
Carnegie Mellon University, Pittsburgh PA, United States 
}
\maketitle              % typeset the header of the contribution
\begin{abstract}
%The abstract should briefly summarize the contents of the paper in
%150--250 words.
Imagine that you are in front of a cockpit with several on--off buttons. If you were to thoroughly test it, you would need to try a prohibitive number of configurations. But since most bugs in practice can be isolated to  interactions among few components,  having tests that cover every possible pairwise configuration is a good start. However, this is a problem that goes from easy to NP--hard as soon as some pairwise configurations are forbidden. In this paper, we revisit unconstrained combinatorial testing with pairwise coverage on binary parameters and contrast it with the constrained case, showing and conjecturing properties that either are upheld or  change from one to the other. In particular, we discuss the extent to which it remains a good idea---and sometimes indeed optimal---to have every button almost as many times on as off to minimize testing. We propose the first exact algorithm based on integer programming and a faster heuristic that often produces optimal solutions, both outperforming or competitive with their baselines.
\end{abstract}

\section{Introduction}

Many safety tests in software and other types of systems involve the identification of faults due to the interaction among components. 
In the simplest setting, we may regard each component $i$ as an on--off switch, 
say as a \emph{parameter} $X_i$ such that $X_i =1$ when that component is on and $X_i=0$ otherwise. 
For every test of the system, we need to assign a value for each of the system parameters. 
We would like to know if something unexpected, which is what we denote as a \emph{fault}, would occur due to the combination of values for such parameters. 
For two distinct components $i$ and $j$, 
we would like to have tests in which both components are on ($X_i=1$ and $X_j=1$), only the first component is on ($X_i=1$ and $X_j=0$), only the second component is on ($X_i=0$ and $X_j=1$), or both components are off ($X_i=0$ and $X_j=0$); 
provided that each of those cases is feasible.
We would like to minimize the number of tests to cover such cases. 

While testing every configuration of the entire system would be impractical as the number of components grow, 
having enough tests to cover every configuration on subsets of limited size does not grow as fast. 
In fact, most faults can be attributed to interactions among of small number of components~\cite{wallace2001medical,kuhn2002software,kuhn2010practical,Softwarefaultinteractionsandimplicationsforsoftwaretesting2004,PseudoExhaustiveTestingforSoftware2006}. 
If a test leads to a fault, then additional tests can help us isolate which specific subset interaction caused that fault~\cite{zeller2002delta}. 
This form of analysis, \emph{combinatorial interaction testing}, has been applied to 
software testing~\cite{kuhn2002software}, circuit verification~\cite{tang1983circuit}, memory correction~\cite{dumer1989memory}, gene regulation~\cite{shasha2001regulation}, materials development~\cite{cawse2003materials}, 
aircraft development~\cite{Aerospace}, 
and machine learning verification~\cite{ma2018combinatorialtestingdeeplearning}. 
We refer the reader to references  \cite{colbourn2004combinatorial,nist_comb_testing,Aerospace,ma2018combinatorialtestingdeeplearning} for more information on applications.

In this work, 
we focus on the \emph{covering array problem},  
which concerns with obtaining sufficient tests to cover every configuration on subsets of limited size, 
hence ensuring that any fault caused by a subset of such size could be detected.  
Although this topic has extensive literature discussed in multiple surveys~\cite{colbourn2004combinatorial,wu2019surveyconstrainedcombinatorialtesting,Grindal2005CombinationTS,Nie2011ASO}, 
comparatively few references focused on obtaining provably optimal solutions~\cite{kleitman1973families,hnich2006cp,danziger2009cafe,maltais2010cafe,MIP2013,kadioglu2017cg,yang2021decomposing} and the only approach to the constrained case is by enumeration~\cite{danziger2009cafe}. 
The presence of pairwise constraints makes combinatorial testing NP-hard~\cite{maltais2010cafe}, 
but constraints preventing specific configurations are often dealt with in an ad-hoc manner by altering the solutions produced for the unconstrained case~\cite{wu2021comparative}. 
The typical approach is by popular families of heuristic algorithms such as IPO~\cite{IPO1998,IPOG2007,IPOGIPOGD2008,IPOGHyper2017,MFT2015,IPOfamilybook2014}, AETG~\cite{AETG1996,Cohen1997TheAS,AETGSAT2008}, orthogonal arrays \cite{Williams2000}, PICT\cite{Czerwonka2006PairwiseTI}, DDA~\cite{bryce2006prioritized,Bryce2007TheDA}, as well as maximum satisfiability~\cite{ansotegui2022maxsat} and metaheuristics~\cite{Cohen2003ConstructingTS,HARTMAN2004149,SA2007,GA2011,CCAG2021}. Meanwhile, Integer Programming~(IP) has only been applied to the unconstrained case~\cite{MIP2013,kadioglu2017cg} or as part of heuristics~\cite{2017ferrer,Ferrer2021}.

To the best of our knowledge, no prior work has described the properties of optimal solutions for the covering array problem beyond the unconstrained case. 
Even then, we believe that a new perspective on the unconstrained case helps understanding what happens to the structure of the optimal solutions when constraints are imposed. 
Hence, our contributions are the following:
\begin{enumerate}[(i)]
\item we revisit the unconstrained binary pairwise covering array problem by focusing on the incidence of \emph{balanced columns} in optimal solutions, 
each implying that a parameter is tested as often with zero and one (Section~\ref{sec:unconstrained}); 
\item we establish how the structure of optimal solutions change, or not, when constraints prevent assignments to pairs of parameters, 
which define what we denote as \emph{constrained pairwise covering problems} (Section~\ref{sec:pairwise}); and 
\item we present and evaluate IP models leveraging our findings in exact as well as heuristic approaches to the constrained case (Sections~\ref{sec:milp} and \ref{sec:results}).
\end{enumerate}

\section{Background and Notation}\label{sec:background}

Consider a system with parameters $X_1, X_2, \ldots, X_k$ and that each parameter is binary. We assign values to these parameters to verify the correct functioning of the system. 
Each such assignment is denoted as a \emph{test}. 
A collection of tests which satisfies a desired verification property is denoted as a \emph{covering array}. 
In terms of what would be desired from such tests, 
trying every possible test ($2^k$ in total) would be impossible for sufficiently large $k$. However, it is technically feasible, and in fact effective to try all configurations on small subsets of parameters. 
In the most common setting, \emph{pairwise covering}, those tests should contain every possible assignment for each of the $\binom{k}{2}$ pairs of parameters.
For two parameters $X_i$ and $X_j$, for example, we should have each pairwise assignment among $(X_i,X_j)=(0,0), (0,1), (1,0),$ and $(1,1)$ in at least one of the tests. The matrix $\BFM^{(1)}$ below, with columns denoting parameters and rows denoting tests, represents an optimal pairwise covering array for $k=10$ unconstrained binary parameters, in the sense that it has the minimum number $N=6$ of tests:
\[
\arraycolsep6pt
\BFM^{(1)} = \begin{pmatrix} 
1 & 1 & 1 & 1 & 1 & 1 & 1 & 1 & 1 & 1 \\ 
0 & 0 & 0 & 1 & 0 & 0 & 1 & 0 & 1 & 1 \\ 
0 & 0 & 1 & 0 & 0 & 1 & 0 & 1 & 0 & 1 \\ 
0 & 1 & 0 & 0 & 1 & 0 & 0 & 1 & 1 & 0 \\ 
1 & 0 & 0 & 0 & 1 & 1 & 1 & 0 & 0 & 0 \\ 
1 & 1 & 1 & 1 & 0 & 0 & 0 & 0 & 0 & 0 
\end{pmatrix}.
\] 
In other words, there are $2^{10}$ distinct assignments on 10 binary parameters, but with 6 of those assignments we covered every pairwise interaction. For example, the test $(X_1, X_2, \ldots, X_{10}) = (0, 0, \ldots , 0 )$ is not part of $\BFM^{(1)}$, but there are rows of zeros in every submatrix on two columns, i.e., each pairwise assignment $(X_1,X_2) = (0,0)$, $(X_1,X_3) = (0,0), \ldots, (X_9,X_{10}) = (0,0)$ occurs at least once. 
Similarly, each such submatrix also covers $(X_i, X_j) = (0,1), (1,0),$ and $(1,1)$.

In the general decision version of the covering array problem, we want to determine if there are $N$ tests on $k$ parameters, each parameter having $v$ values, that define a covering array of \emph{strength} $t$. If so, then every possible assignment for each tuple of $t$ out of the $k$ parameters occurs in some test. 
Following the conventions in the literature, 
let $\CA(N;t,k,v)$ be the set of $N \times k$ matrices, or covering arrays, in which every submatrix on $t$ columns includes all $t-$tuples over $\{0, \ldots, v-1\}^t$ at least once; and 
let $\CAN(t,k,v)$ be the minimum $N$ such that $\CA(N;t,k,v) \neq \emptyset$.  
In this paper, 
we also introduce some additional notation: 
we use $\BFM_i$ for the $i$-th column of $\BFM \in \CA(N;t,k,v)$, 
i.e., the values assigned to $X_i$; 
and we define $\CAS(t,k,v)$ as the set $\CA(N;t,k,v)$ in which $N=\CAN(t,k,v)$.

The most commonly studied case consists of $v=2$ (binary parameters) and $t=2$ (pairwise covering), 
for which the optimal number of tests is known: 
$\CAN(2,k,2) = N$, where $N$ is the smallest integer such that 
$k \leq \binom{N-1}{\lceil N/2 \rceil}$~\cite{renyi1970foundations,katona1973applications,kleitman1973families}. 
Hence, $\kmax(N) := \binom{N-1}{\lceil N/2 \rceil}$ is the largest number of binary parameters for which $N$ tests may ensure pairwise covering. 
Bounds on the number of tests $N$ for other families of cases involving $t,k,$ and $v$ can be found in \cite{colbourn2004combinatorial} and \cite{lawrence2011survey}.

For $N=6$, $\kmax(N)=10$ as in matrix $\BFM^{(1)}$. 
In fact, we have obtained $\BFM^{(1)}$ with a commonly used construction technique by Kleitman and Spencer~\cite{kleitman1973families}, 
which consists of having a row in which all values are one and the remaining rows are such that each column is distinct and has $\lfloor \frac{N}{2}\rfloor -1$ other elements having value one. 
More generally,  
we discuss along the paper that 
the presence of a uniform row, or the equivalent of such a row, as well as  
having columns with a similar number of zeros and ones  are important for characterizing optimal solutions. Those are still relevant when pairwise constraints apply to the system.

In this paper, we consider pairwise covering problems in which some assignments are forbidden. 
For example, we may have $(X_i, X_j) \neq (0,0)$ for a specific pair of parameters, 
in which case we must cover $(X_i, X_j) = (0,1), (1,0),$ and $(1,1)$ across the tests, but $(X_i, X_j) = (0,0)$ must never occur. 
We note that having more than one forbidden assignment for the same pair of parameters actually simplifies the covering array problem in one way or another. 
When there are two forbidden assignments, 
we have one of the following cases:
\[
\begin{array}{ccc}
(X_i, X_j) \neq (\alpha, \beta) & \wedge & 
\left\{
\begin{array}{ccc}
(X_i, X_j) \neq (\alpha, 1-\beta) & \rightarrow & X_i = 1 - \alpha. \\
(X_i, X_j) \neq (1-\alpha, \beta) & \rightarrow & X_j = 1 - \beta. \\
(X_i, X_j) \neq (1-\alpha, 1-\beta) & \rightarrow & 
\left \{ 
\begin{array}{ccc}
X_i = X_j & \text{ if } & \alpha \neq \beta. \\
X_i = 1- X_j & \text{ if } & \alpha = \beta.
\end{array}
\right.
\end{array}
\right.
\end{array}
\]
Therefore, 
having two forbidden assignments for a given pair of parameters implies that 
either one of the parameters is fixed (first two deductions), 
or that the value for one parameter is dependent on the other (last two deductions); 
and in both cases we can reduce the problem by removing one of those parameters. 
With three forbidden assignments,
say all but $(X_i, X_j) = (\alpha, \beta)$ forbidden,
then $X_i = \alpha$ and $X_j = \beta$. 
Finally, and obviously, the covering array problem is infeasible with four forbidden assignments on the same pair of parameters. 
Hence, without loss of generality, we assume the following simplifying conditions:
\newtheorem{assumption}{Assumption}
\begin{assumption}\label{ass:first}
No parameter has a fixed value.
\end{assumption}
\begin{assumption}
There is at most one constraint on every pair of parameters.
\end{assumption}
\begin{assumption}\label{ass:last}
There are no parameters with an implicitly fixed value, 
nor a pair of parameters with an implicit constraint other than those known in advance.
\end{assumption}
In practice, 
we can remove parameters with fixed or implied values in preprocessing 
and then reintroduce them with the corresponding values in postprocessing. 
Likewise, 
we can generate all implied constraints with the deductive steps as outlined in~\cite{MFT2014,MFT2015}. 
For example, $(X_1,X_2) \neq (1,0) \wedge (X_2,X_3) \neq (1,0) \rightarrow (X_1,X_3) \neq (1,0)$.
Hence, Assumptions~\ref{ass:first} to \ref{ass:last} are also not restricting the scope of our work.

We will avoid an implicit dependence on the assumptions above in our theoretical results in Section~\ref{sec:pairwise} by using the following definition: 

\begin{definition}[Reduced Constrained Pairwise Covering Problem]\label{def:reduced}
    Let a constrained pairwise covering problem be \emph{reduced} when it has (i) no fixed parameter values, and (ii) at most one pairwise constraint per pair of parameters; be those explicitly defined or implied by the other constraints.
\end{definition}

Effectively, Definition~\ref{def:reduced} characterizes a constrained pairwise covering problem to which the preprocessing previously described has already been applied.

Although there are always constraints in the pairwise covering array problem, 
since every pair of assignments should be either covered by a test or forbidden in all tests, 
we refer to this problem as being \emph{unconstrained} if no specific interactions are forbidden. 
Our study starts by revisiting the unconstrained case.

\section{Revisiting the Unconstrained Case}\label{sec:unconstrained}

There are some properties that characterize many---and, in some cases, all---optimal solutions for the unconstrained covering array problem. 
By first understanding and reframing those in Section~\ref{sec:unconstrained}, 
we will later generalize such properties to circumvent forbidden interactions in a systematic way in Sections~\ref{sec:pairwise} and \ref{sec:milp}.

We continue using $k$ and $N$ to denote the number of parameters and tests.

\subsection{Four Simple Facts About Unconstrained Covering Arrays}\label{sec:feasibility}

The propositions below are possibly already known in one form or another, 
but they are certainly useful.
We start with necessary conditions for a matrix to be a covering array. 
Namely, all the columns should be unique and not the complement of one another. 
We formalize those conditions as follows:

\begin{proposition}[Uniqueness of Columns] \label{prop:unique_cols}
Let $k, N \in \mathbb{Z}_+$,  $N \geq \CAN(2,k,2)$. 
Any two columns $\BFM_i, \BFM_j \in \{0,1\}^N$, $i \neq j$, of a matrix $\BFM \in \CA(N;2,k,2)$ are such that $\BFM_i \neq \BFM_j$. 
\end{proposition}
\begin{proof}%[Proposition~\ref{prop:unique_cols}]
If $\BFM_i = \BFM_j$, then $\BFM$ would not cover the assignments $(X_i, X_j) = (0,1)$ and $(X_i, X_j) = (1,0)$. Hence, no two columns are the same. \qed 
\end{proof}

\begin{proposition}[Exclusion of Complement] \label{prop:not_complement}
Let $k, N \in \mathbb{Z}_+$,  $N \geq \CAN(2,k,2)$. 
Any two columns $\BFM_i, \BFM_j \in \{0,1\}^N$, $i \neq j$, of a matrix $\BFM \in \CA(N;2,k,2)$ are such that $\BFM_i \neq \BFone -\BFM_j$, $\BFone = \{1\}^N$. 
\end{proposition}
\begin{proof}%[Proposition~\ref{prop:not_complement}]
If $\BFM_i = \BFone - \BFM_j$, then $\BFM$ would not cover the assignments $(X_i, X_j) = (0,0)$ and $(X_i, X_j) = (1,1)$. Hence, no two columns are  complements.
\qed
\end{proof}

Moreover, we can obtain alternate covering arrays by taking the complement of columns or swapping them. 
These properties will be useful later: 

\begin{proposition}[Covering by Column Complement] \label{prop:alternate_complement}
Let $k, N \in \mathbb{Z}_+$,  $N \geq \CAN(2,k,2)$. 
For any $\BFM \in \CA(N;2,k,2)$ and column $\BFM_i$, 
let matrix $\BFM'$ be such that $\BFM'_i = \BFone - \BFM_i$, $\BFone = \{1\}^N$, and $\BFM'_j = \BFM_j ~ \forall j \neq i$. 
Then $\BFM' \in \CA(N;2,k,2)$.
\end{proposition}
\begin{proof}
Since $\BFM$ and $\BFM'$ are identical except for column $i$, then all pairwise assignments not involving parameter $X_i$ are covered by $\BFM'$ if they are covered by $\BFM$. 
For any row of $\BFM$ covering $(X_i,X_j) = (\alpha, \beta)$, the same row in $\BFM'$ covers  $(X_i,X_j) = (1-\alpha, \beta)$. 
Thus, if there is a row in $\BFM$ covering each assignment of $(X_i,X_j)$, then there is another row in $\BFM'$ covering each of those assignments. 
Hence, taking the complement of a column yields another covering array.
\qed
\end{proof}

For an application of Proposition~\ref{prop:alternate_complement}, 
consider the matrices $\BFM^{(2)}$ and $\BFM^{(3)}$:
\[
\arraycolsep6pt
\BFM^{(2)} = \begin{pmatrix} 
1 & 1 & 1 & 1 \\ 
0 & 0 & 0 & 1 \\ 
0 & 0 & 1 & 0 \\ 
0 & 1 & 0 & 0 \\ 
1 & 0 & 0 & 0 
\end{pmatrix},~ 
\BFM^{(3)} = \begin{pmatrix} 
1 & 1 & 1 & 0 \\ 
0 & 0 & 0 & 0 \\ 
0 & 0 & 1 & 1 \\ 
0 & 1 & 0 & 1 \\ 
1 & 0 & 0 & 1 
\end{pmatrix}.
\] 
Both of those matrices represent optimal pairwise covering arrays for $k=5$. Similarly to $\BFM^{(1)}$, the matrix $\BFM^{(2)}$ is obtained with the construction technique by Kleitman and Spencer~\cite{kleitman1973families}. 
In turn, 
the matrix $\BFM^{(3)}$ is obtained based on Proposition~\ref{prop:alternate_complement} by taking the complement of the last column of $\BFM^{(2)}$.

\begin{proposition}[Covering by Column Swapping] \label{prop:multiplicity}
Let $k, N \in \mathbb{Z}_+$,  $N \geq \CAN(2,k,2)$. 
For any $\BFM \in \CA(N;2,k,2)$ and columns $\BFM_{i_1}$ and $\BFM_{i_2}$, 
let matrix $\BFM'$ be such that $\BFM'_{i_1} = \BFM_{i_2}$, $\BFM'_{i_2} = \BFM_{i_1}$, and $\BFM'_j = \BFM_j ~ \forall j \notin {i_1, i_2}$. 
Then $\BFM' \in \CA(N;2,k,2)$.
\end{proposition}
\begin{proof}
Since $\BFM$ and $\BFM'$ are identical except for columns $i_1$ and $i_2$, then all pairwise assignments not involving either parameter $X_{i_1}$ nor $X_{i_2}$ are covered by $\BFM'$ if they are covered by $\BFM$. 
Every assignment $(X_{i_1},X_{i_2}) = (\alpha, \beta)$ is covered in $\BFM'$ due to the corresponding assignment 
$(X_{i_1},X_{i_2}) = (\beta, \alpha)$ being covered in $\BFM$. 
Finally, for $i_A \in \{i_1, i_2\}$ (i.e., one of the indices among $i_1$ and $i_2$), $i_B \in \{i_1, i_2\} \setminus \{ i \}$ (i.e., the other index), and $j \notin \{i_1, i_2\}$,  
every assignment $(X_i,X_j) = (\alpha, \beta)$ is covered in $\BFM'$ due to assignment 
$(X_{\bar{i}},X_j) = (\alpha, \beta)$ being covered in $\BFM$.
Hence, swapping any two columns yields another covering array.
\qed
\end{proof}

With Proposition~\ref{prop:multiplicity}, the columns of an unconstrained covering array are interchangeable. 
Along with Proposition~\ref{prop:alternate_complement}, that implies that every covering array on $k$ parameters belongs to an equivalence class of $2^k k!$ covering arrays upon column complement and column swapping. 
To the extent that we can retain part of this flexibility even after considering pairwise constraints breaking such symmetry, 
we can simplify the problem of finding an optimal covering array. 
That will be particularly helpful in Section~\ref{sec:milp}.

\subsection{Even Covering Arrays: The Special Case}\label{sec:even}

When we have an even number of rows, and thus each column can have as many zeros as ones, 
we often observe optimal covering arrays that reflect that. 
For example, matrix $\BFM^{(1)}$ has 3 zeros and 3 ones in every column.
In fact, we can rely on the following sufficiency condition when the number of rows $N$ is even:

\begin{lemma}[Balanced Covers] \label{lem:balanced}
Let $k, N \in \mathbb{Z}_+$ with $N$ even. 
Any matrix $\BFM \in \{0, 1\}^{N \times k}$ in which each column $\BFM_i$ is (a) unique ($\BFM_i \neq \BFM_j ~\forall j \neq i$); (b) not the complement of another column ($\BFM_i \neq \BFone - \BFM_j, \BFone = \{1\}^N ~\forall j \neq i$); and (c) balanced (has $\frac{N}{2}$ zeros and $\frac{N}{2}$ ones) is such that $\BFM \in \CA(N;2,k,2)$.
\end{lemma}
\begin{proof}%[Lemma~\ref{lem:balanced}]
Let us suppose, for contradiction, that there is a matrix $\BFM$ satisfying conditions (a)--(c), 
but such that $\BFM \notin \CA(N;2,k,2)$. 
Hence, there should be columns $\BFM_i$ and $\BFM_j$, $i \neq j$, and a pair of values $(\alpha, \beta) \in \{0, 1\}^2$ such that there is no row in $\BFM$ covering $(X_i, X_j) = (\alpha, \beta)$. 
We show that this is not possible.

If that were the case for $\alpha \neq \beta$, then column $\BFM_i$ having $\alpha$ in a row implies that column $\BFM_j$ has $\alpha$ in that row. 
Since columns $\BFM_i$ and $\BFM_j$ have the same number of zeros and ones due to (c),
then $\BFM_i$ and $\BFM_j$ would have $\alpha$ in the same rows 
and consequently $\beta$ in the same rows. 
Thus, $\BFM_i = \BFM_j$, contradicting (a). 

If that were the case for $\alpha = \beta$, then column $\BFM_i$ having $\alpha$ in a row implies that column $\BFM_j$ has $1 - \alpha$ in that row.
Since $\BFM_i$ and $\BFM_j$ have as many zeros as ones due to (c), 
then column $\BFM_i$ would have $\alpha$ in the same rows that column $\BFM_j$ would have $1 - \alpha$ and vice versa. 
Thus, $\BFM_i = \BFone - \BFM_j$, contradicting (b).
\qed
\end{proof}

These results outline a simple algorithm that does not require the conventional use of a row of ones for constructing covering arrays for a given choice of $k$ parameters using an even number of $N$ tests, which is based on three rules:
\begin{enumerate}
    \item Generate columns without repetition (Proposition~\ref{prop:unique_cols}).
    \item Do not generate the complement of a generated column (Propositon~\ref{prop:not_complement}). 
    \item Only generate columns with equal number of zeros and ones (Lemma~\ref{lem:balanced}).
\end{enumerate}
In fact, this algorithm produces an optimal solution for the smallest possible $N$, 
if such an $N$ is even, 
given that  $\kmax(N) = \binom{N-1}{\lceil N/2 \rceil}$ (see Section~\ref{sec:background}): 

\begin{corollary}[Even Balanced Optimality]\label{cor:even_optimality}
Let $k \in \mathbb{Z}_+$. 
If $N := \CAN(2,k,2)$ is even, we obtain a matrix in $\CA^*(2,k,2)$ by following rules 1, 2, and 3.
\end{corollary}
\begin{proof}%[Corollary~\ref{cor:even_optimality}]
The number of binary columns with $N/2$ ones is $\binom{N}{N/2}$. 
If we deduct the complement of each column, we have $\frac{1}{2} \binom{N}{N/2} = \frac{1}{2} \frac{N!}{N/2! N/2!}= \frac{1}{2} \frac{N (N-1)!}{N/2 (N/2-1)! N/2!} = \frac{(N-1)!}{(N/2-1)! N/2!} = \binom{N-1}{N/2} = \kmax(N)$ columns.
\qed
\end{proof}

\subsection{Odd Covering Arrays: The General Case}\label{sec:odd}

In the case of $N$ odd, working with the most balanced columns would amount to considering columns with $\lfloor N/2 \rfloor$ ones and $\lceil N/2 \rceil$ zeros and vice versa. 
For example, consider the two previous matrices: $\BFM^{(2)}$ has 2 ones and 3 zeros in every column; whereas $\BFM^{(3)}$ has the last column with 3 ones and 2 zeros.

Since to every column with $\lfloor N/2 \rfloor$ ones and $\lceil N/2 \rceil$ zeros 
there is a complement with $\lceil N/2 \rceil$ ones and $\lfloor N/2 \rfloor$ zeros, 
then by Proposition~\ref{prop:alternate_complement} we can use each of such pairs of columns interchangeably. 
Without loss of generality, 
we make the following assumption while discussing the upcoming  Propositions~\ref{prop:gen_zero_cover} to \ref{prop:one_row}  
to make it easier to understand the structure of covering arrays with $N$ odd:

\begin{assumption}\label{ass:odd}
For $N$ odd, \emph{balanced columns} those with $\lfloor N/2 \rfloor$ ones and $\lceil N/2 \rceil$ zeros.
\end{assumption}

We start by proving that any matrix with distinct columns of $\lfloor N/2 \rfloor$ ones and $\lceil N/2 \rceil$ zeros covers $(X_i, X_j) = (0,0)$, $(X_i, X_j) = (0,1)$, and $(X_i, X_j) = (1,0)$ 
with the following results. 
Please note that they also hold for $N$ even, 
which is a convenient generalization for reusing these results in the constrained case:

\begin{proposition}[Generalized Zero-Zero Cover]\label{prop:gen_zero_cover}
Let $k, N \in \mathbb{Z}_+$. 
Any matrix $\BFM \in \{0, 1\}^{N \times k}$ in which each column (a) is unique; (b) is not the complement of another column; and (c) has at most $\lfloor N/2 \rfloor$ ones (resp. at least $\lceil N/2 \rceil$ zeros) is such that $(X_i, X_j) = (0,0)$ is covered in least one row for any $(i, j) \in \{1, \ldots, k\}^2$. 
\end{proposition}
\begin{proof}%[Proposition~\ref{prop:gen_zero_cover}]
Each column having at most $\lfloor N/2 \rfloor$ ones implies that there are at most $2 \lfloor N/2 \rfloor$ rows with ones in either $\BFM_i$ or $\BFM_j$. Since  $2 \lfloor N/2 \rfloor = N$ only for $N$ even and with both columns balanced, 
then the conditions of Lemma~\ref{lem:balanced} apply in such case. 
Otherwise, if $2 \lfloor N/2 \rfloor < N$, 
then at least one row has a pair of zeros. 
\qed
\end{proof}

\begin{proposition}[Generalized Zero-One Cover]\label{prop:gen_zero_one_cover}
Let $k, N \in \mathbb{Z}_+$. 
A matrix $\BFM \in \{0, 1\}^{N \times k}$ covers $(X_i, X_j) = (0,1)$ in at least one row for any $(i, j) \in \{1, \ldots, k\}^2, i \neq j$ if, and only if, 
column $\BFM_j$ is not a subset of $\BFM_i$ ($\BFM_j \nleq \BFM_i$).
\end{proposition}
\begin{proof}%[Proposition~\ref{prop:gen_zero_one_cover}]
If $\BFM_j \nleq \BFM_i$, 
then is at least one row in which the value assigned to $X_j$ is greater than the value assigned to $X_i$, 
which thus covers $(X_i, X_j) = (0, 1)$.

Conversely, if $\BFM$ covers $(X_i, X_j) = (0, 1)$, then the rows with ones in column $\BFM_j$ are not a subset of the rows with ones in $\BFM_i$, and thus $\BFM_j \nleq \BFM_i$.
\qed
\end{proof}

In other words, 
by assuming that balanced columns may have more zeros than ones,  
we have made zeros slightly more abundant. 
Conveniently, 
that entails that all pairwise assignments involving the value zero are covered. 
Hence, 
we are left to discuss how to cover pairs of ones, given that ones are less abundant.

On the one hand, 
if we adapt the third rule in Section~\ref{sec:even} to handle $N$ odd by assuming as balanced columns those with $\lfloor N / 2 \rfloor$ ones,  
we may not always find an optimal solution for the smallest $N$. 
For example, consider again matrix $\BFM^{(2)}$, repeated below for convenience, along with another matrix $\BFM^{(4)}$:
\[
\arraycolsep6pt
\BFM^{(2)} = \begin{pmatrix} 1 & 1 & 1 & 1 \\ 0 & 0 & 0 & 1 \\ 0 & 0 & 1 & 0 \\ 0 & 1 & 0 & 0 \\ 1 & 0 & 0 & 0 \end{pmatrix},~ 
\BFM^{(4)} = \begin{pmatrix} 0 & 0 & 0 \\ 0 & 0 & 0 \\ 0 & 1 & 1 \\ 1 & 0 & 1 \\ 1 & 1 & 0 \end{pmatrix}.
\] 
We cover 4 parameters with the leftmost matrix $\BFM^{(2)} \in \CA(5;2,4,2)$, 
and thus we cover 3 parameters with any submatrix of $\BFM^{(2)}$ on three columns. 
We also cover 
3 parameters with the rightmost matrix $\BFM^{(4)} \in \CA(5;2,3,2)$. 
However, there is no column that could be added to $\BFM^{(4)}$ to cover 4 parameters. 
Therefore, depending on which columns we choose first, 
we may not be able to extend a covering array on $k$ parameters and $N$ tests all the way to $\kmax(N)$ parameters. 

On the other hand, based on Propositions~\ref{prop:gen_zero_cover} and \ref{prop:gen_zero_one_cover}, our only concern is about covering $(X_i, X_j) = (1, 1)$. 
This can be achieved by having a row of ones, as in the construction used for $\BFM^{(1)}$ and $\BFM^{(2)}$, 
implying a simple sufficient result:

\begin{proposition}[One-Row Sufficiency]\label{prop:one_row}
Let $k, N \in \mathbb{Z}_+$. 
Any matrix $\BFM \in \{0, 1\}^{N \times k}$ in which each column is (a) unique and (b) has $\lfloor N/2 \rfloor$ ones and $\lceil N/2 \rceil$ zeros; and (c) one row has $k$ ones is such that $\BFM \in \CA(N;2,k,2)$. 
\end{proposition}
\begin{proof}%[Proposition~\ref{prop:one_row}]
Condition (c) implies covering all assignments of form $(X_i, X_j) = (1,1)$, and (a)--(b) imply covering other assignments due to Propositions~\ref{prop:gen_zero_cover} and \ref{prop:gen_zero_one_cover}. 
\qed
\end{proof}

Now we take a step back to drop\footnote{For the sake of generality, we never explicitly referred to columns with $\lfloor N/2 \rfloor$ ones and $\lceil N/2 \rceil$ zeros as balanced in Propositions~\ref{prop:gen_zero_cover} to \ref{prop:one_row}. However, we primed the reader to think of those as balanced columns in advance to help communicate the results.} Assumption~\ref{ass:odd} and reframe the result above to consider the general case of mixing columns that have $\lfloor N/2 \rfloor$ ones and $\lceil N/2 \rceil$ zeros 
with columns that have $\lceil N/2 \rceil$ ones and $\lfloor N/2 \rfloor$ zeros:

\begin{corollary}[Least-Row Sufficiency]\label{cor:least_row}
Let $k, N \in \mathbb{Z}_+$. 
Any matrix $\BFM \in \{0, 1\}^{N \times k}$ in which each column (a) is unique; (b) is not the complement of another column; (c) has at least $\lfloor N/2 \rfloor$ ones and $\lfloor N/2 \rfloor$ zeros; and (d) has a row with the least frequent element of each column is such that $\BFM \in \CA(N;2,k,2)$. 
\end{corollary}
\begin{proof}
By taking the complement of each column with more ones than zeros, 
we obtain a matrix $\BFM' \in \{0, 1\}^{N \times k}$ satisfying conditions (a), (b), and (c) of Proposition~\ref{prop:one_row}. 
Due to Proposition~\ref{prop:alternate_complement}, it follows that $\BFM \in \CA(N;2,k,2)$.
\end{proof}

In other words, 
we ensure that we have a covering array 
by having a row with the element that is slightly less frequent in each of the columns. 

Because the sufficiency condition from Proposition~\ref{prop:one_row} taps back to the construction by Kleitman and Spencer~\cite{kleitman1973families}, 
the optimality is straightforward when $N = \CAN(2,k,2)$ is odd:  
after deducting the row of ones, 
we are left with $N-1$ rows of which $\lceil N / 2 \rceil$ are zeros, allowing for $\binom{N-1}{\lceil N / 2 \rceil} = \kmax(N)$ distinct columns. 
Nevertheless, the necessary conditions for covering $(X_i, X_j) = (\alpha, \beta)$ when $\alpha \neq \beta$ (Proposition~\ref{prop:gen_zero_one_cover} applied both ways) and the sufficient conditions for when $\alpha = \beta$ (Propositions~\ref{prop:gen_zero_cover} and \ref{prop:one_row}) are helpful to tackle the constrained case.

\subsection{Optimal Covering Arrays with Unbalanced Columns}\label{sec:shadow}

We have seen that there is always an optimal covering array in which all columns are balanced. However, as one would expect, 
there might be other optimal covering arrays; 
at least in some cases. 
For example, consider matrix $\BFM^{(5)}$:
\[
\arraycolsep6pt
\BFM^{(5)} = \begin{pmatrix} 
1 & 1 & 1 & 1 & 1 & 1 & \textbf{1} \\ 
0 & 0 & 0 & 0 & 0 & 0 & \textbf{1} \\ 
0 & 0 & 1 & 0 & 1 & 1 & \textbf{0} \\ 
0 & 1 & 0 & 1 & 0 & 1 & \textbf{0} \\ 
1 & 0 & 0 & 1 & 1 & 0 & \textbf{0} \\ 
1 & 1 & 1 & 0 & 0 & 0 & \textbf{0} 
\end{pmatrix}.
\] 
Matrix $\BFM^{(5)}$ is a covering array for 7 parameters, a case that requires at least 6 tests, and thus $\BFM^{(5)}$ is optimal. You may notice that the last column of $\BFM^{(5)}$, in bold, is not balanced. 
However, it is not always the case that we can find covering arrays with unbalanced columns for any number of parameters. 

We will use this example to provide a brief but more general intuition for the role of balanced columns in optimal covering arrays. 
First, notice that matrix $\BFM^{(5)}$ can be obtained from $\BFM^{(1)}$. We begin by changing the third element of the last column of $\BFM^{(1)}$, hence replacing column $[1 ~ 1 ~ 1 ~ 0 ~ 0 ~ 0]^T$ with $[1 ~ 1 ~ \textbf{0} ~ 0 ~ 0 ~ 0]^T$. 
Because only the first two elements in the last column are one, 
we must necessarily have both a zero and a one in the first two rows of all other columns to have a covering array. 
Hence, we must drop columns 4, 7, and 9. Starting from $\BFM^{(1)}$, we show below the changed element in the last column in bold and the removed columns in gray, so that the columns in black correspond to $\BFM^{(5)}$:
\[
\arraycolsep6pt
\begin{pmatrix} 
1 & 1 & 1 & {\color{gray} 1} & 1 & 1 & {\color{gray} 1} & 1 & {\color{gray} 1} & 1 \\ 
0 & 0 & 0 & {\color{gray} 1} & 0 & 0 & {\color{gray} 1} & 0 & {\color{gray} 1} & 1 \\ 
0 & 0 & 1 & {\color{gray} 0} & 0 & 1 & {\color{gray} 0} & 1 & {\color{gray} 0} & \textbf{0} \\ 
0 & 1 & 0 & {\color{gray} 0} & 1 & 0 & {\color{gray} 0} & 1 & {\color{gray} 1} & 0 \\ 
1 & 0 & 0 & {\color{gray} 0} & 1 & 1 & {\color{gray} 1} & 0 & {\color{gray} 0} & 0 \\ 
1 & 1 & 1 & {\color{gray} 1} & 0 & 0 & {\color{gray} 0} & 0 & {\color{gray} 0} & 0 
\end{pmatrix}.
\] 

Second, consider how the last column prevents us from adding any other column to $\BFM^{(5)}$ and still obtain a valid covering array. 
As we have seen above, a new column would necessarily need at least a zero and a one in the first two rows. 
Without loss of generality, we may assume that the value one is in the first row and the value zero in the second row, since we can otherwise take the complement by Preposition~\ref{prop:alternate_complement}.  
Hence, any new column would necessarily have one in the first row and zero in the second row. 
However, 
a new column could not have two ones in the remaining rows since otherwise it would be identical to one of the existing columns, 
contradicting either Proposition~\ref{prop:unique_cols} or that we have a valid covering array. 
Likewise, 
a new column with more or less than two ones would imply a containment ($\leq$) relationship with one of the existing columns, 
contradicting either Proposition~\ref{prop:gen_zero_one_cover} or that we have a valid covering array. 

More generally, 
using an unbalanced column with fewer (or more) ones in a covering array 
prevents us from using other columns that are supersets (or subsets) of that column. 
When it comes to balanced columns, 
that means trading many balanced columns for an unbalanced one. 
For example, 
by using the unbalanced column $[1 ~ 1 ~ 0 ~ 0 ~ 0 ~ 0]^T$, 
we could not use columns $[1 ~ 1 ~ 0 ~ 0 ~ 0 ~ 1]^T$, $[1 ~ 1 ~ 0 ~ 0 ~ 1 ~ 0]^T$, $[1 ~ 1 ~ 0 ~ 1 ~ 0 ~ 0]^T$, $[1 ~ 1 ~ 1 ~ 0 ~ 0 ~ 0]^T$, and their complements.
Hence, while it is possible that some optimal covering arrays with $N$ tests have unbalanced columns, 
that possibility vanishes as the number of parameters approaches $\kmax(N)$.

\section{Optimization with Pairwise Constraints}\label{sec:pairwise}

So far we have assumed that all pairs of assignments must be covered. 
But if there is a forbidden assignment that must be avoided, say $(X_i, X_j) \neq (\alpha, \beta)$, 
then we need to adapt the constructions discussed in Section~\ref{sec:unconstrained} accordingly. 
In this case, we must cover all but the pairs of assignments which are forbidden.

For example, what if we have $k = 5$ parameters subject to $(X_4, X_5) \neq (1, 0)$?
With $N = 6$ rows as in the unconstrained case, we should not use balanced columns for both $X_4$ and $X_5$, 
since otherwise we would automatically cover $(X_4, X_5) = (1, 0)$. 
The solution changes even more if we add more constraints, such as $(X_3, X_5) \neq (1, 0)$ and $(X_2, X_5) \neq (1, 0)$. 
Consider matrices $\BFM^{(6)}$ to $\BFM^{(9)}$:
\[
\begingroup
\setlength\arraycolsep{1pt}
\BFM^{(6)}{=}\begin{pmatrix} 
1 & 1 & 1 & 1 & 1 \\ 
1 & 0 & 0 & 0 & 0 \\ 
0 & 1 & 0 & 0 & 0 \\ 
0 & 0 & 1 & 0 & 1 \\ 
0 & 0 & 0 & 1 & 1 \\ 
1 & 1 & 1 & 1 & 0 
\end{pmatrix}, 
\BFM^{(7)}{=}\begin{pmatrix} 
1 & 1 & 1 & 1 & 1 \\ 
1 & 0 & 0 & 0 & 0 \\ 
0 & 1 & 0 & 0 & 0 \\ 
0 & 0 & 1 & 0 & 1 \\ 
0 & 0 & 0 & 1 & 1 \\ 
1 & 1 & 1 & \text{\textbf{0}} & 0
\end{pmatrix},  
\BFM^{(8)}{=}\begin{pmatrix} 
1 & 1 & 1 & 1 & 1 \\ 
1 & 0 & 0 & 0 & 0 \\ 
0 & 1 & 0 & 0 & 0 \\ 
0 & 0 & 1 & 0 & 1 \\ 
0 & 0 & 0 & 1 & 1
\end{pmatrix},  
\BFM^{(9)}{=}\begin{pmatrix} 
1 & 1 & 1 & 1 & 1 \\ 
1 & 0 & 0 & 0 & 0 \\ 
0 & 1 & 0 & \text{\textbf{1}} & \text{\textbf{1}} \\ 
0 & 0 & 1 & \text{\textbf{1}} & 1 \\ 
0 & 0 & 0 & \text{\textbf{0}} & \text{\textbf{0}} \\ 
1 & 1 & 1 & \text{\textbf{0}} & \text{\textbf{1}} 
\end{pmatrix}{.}
\endgroup
\]
Matrix $\BFM^{(6)}$ is an optimal unconstrained covering array with $k=5$. 
In turn, matrix~$\BFM^{(7)}$ is an optimal\footnote{The claim that the solutions presented for the constrained case are optimal have been validated with either the theory or the algorithms discussed later in the paper.} covering array subject to $(X_4, X_5) \neq (1, 0)$. 
By also adding $(X_3, X_5) \neq (1, 0)$, 
matrix~$\BFM^{(8)}$ is an optimal covering array with surprisingly fewer tests than before. 
By further adding $(X_2, X_5) \neq (1, 0)$, 
matrix~$\BFM^{(9)}$ is now an optimal covering array and 
the number of tests goes back to six. 
We put in bold the elements in matrices~$\BFM^{(7)}$ to $\BFM^{(9)}$ that differ from $\BFM^{(6)}$. 

Those examples highlight important distinctions from the unconstrained case, 
which we enumerate below. They form the basis for the remainder of the paper:

\begin{enumerate}[(i)]
\item The columns are not---and cannot---be all balanced: 
the fourth column of $\BFM^{(5)}$ has more zeros than ones, 
and the fifth column of both $\BFM^{(6)}$ and $\BFM^{(7)}$ have more ones than zeros. 
We present a same underlying principle for both cases in Section~\ref{sec:necessary}. 

\item 
The minimum number of tests may increase (as from $\BFM^{(8)}$ to $\BFM^{(9)}$) or decrease (as from $\BFM^{(7)}$ to $\BFM^{(8)}$) as we add more pairwise constraints. 
(But how is that even possible?) 
We discuss the factors involved in Section~\ref{sec:lb}.

\item The optimal covering arrays still seem to have as many balanced---or closer to balanced---columns as possible, save some notable exceptions. 
We discuss circumstances which may exclude some or all columns from being balanced in optimal covering arrays if their parameters are subject to pairwise constraints in Section~\ref{sec:cliques}. 
In contrast, 
we state a conjecture based on the observation that it is apparently possible to find valid covering arrays having balanced columns for unconstrained parameters in Section~\ref{sec:conjecture}.

\item
The column chosen for one parameter may further affect the choice for another parameter if either needs an unbalanced column. 
For example, the fourth column of $\BFM^{(9)}$ is a balanced column,
but it is not the same balanced column used in $\BFM^{(6)}$ because that column is not compatible with the fifth column of $\BFM^{(9)}$. 
We return to this discussion in Section~\ref{sec:milp}.

\end{enumerate}

\subsection{Necessary Conditions}\label{sec:necessary}

We present the following result as a unifying principle for the number of zeros and ones differing slightly across columns of matrices $\BFM^{(6)}$ to $\BFM^{(9)}$: 
\begin{lemma}[$1$--$0$ Constraint]\label{lem:xi_1_xj_0}
If a reduced constrained pairwise covering problem\footnote{We remind the reader of Definition~\ref{def:reduced} in Section~\ref{sec:background} for our assumptions in this work.} is subject to 
$(X_i, X_j) \neq (1,0)$, 
then in any valid covering array $\BFM \in \{0,1\}^{N \times k}$ 
the column $\BFM_i$  has more zeros---and less ones---than $\BFM_j$.
\end{lemma}
\begin{proof}%[Lemma~\ref{lem:xi_1_xj_0}]
If $(X_i, X_j) = (1, 0)$ should not be covered, 
then $X_i=1 \rightarrow X_j=1$ in all tests.
But if $(X_i, X_j) = (0, 1)$ should be covered in some test, 
then $X_j = 1$ occurs more often than $X_i = 1$; 
and, conversely, $X_i = 0$ more often than $X_j = 0$. 
Hence, there are relatively more ones in $\BFM_j$ and therefore more zeros in $\BFM_i$. 
\qed
\end{proof}

In other words, 
if $X_i$ and $X_j$ are subject only to $(X_i, X_j) \neq (1, 0)$, 
any covering array $\BFM$ has more zeros in $\BFM_i$ (e.g., fourth column of $\BFM^{(7)}$) or more ones in $\BFM_j$ (e.g., fifth column of both $\BFM^{(8)}$ and $\BFM^{(9)}$).

Now we generalize Lemma~\ref{lem:xi_1_xj_0} to lay out necessary conditions for pairwise covering subject to pairwise constraints with any pair of values:

\begin{theorem}[$\alpha$--$\beta$ Constraints]\label{cor:xi_alpha_xj_beta}
For $w_i$ and $w_j$ as the number of ones of parameters $X_i$ and $X_j$ in a reduced constrained pairwise covering problem, then:  
\begin{align}%{ccl}
(X_i, X_j) \neq (0,0) ~  & \rightarrow ~  & w_i + w_j \geq N + 1. \label{eq:00} \\
(X_i, X_j) \neq (0,1) ~ & \rightarrow ~ & w_i - w_j \geq 1. \label{eq:01} \\
(X_i, X_j) \neq (1,0) ~ & \rightarrow ~ & w_j - w_i \geq 1. \label{eq:10} \\
(X_i, X_j) \neq (1,1) ~ & \rightarrow ~ & w_i + w_j \leq N - 1. \label{eq:11}
\end{align}
\end{theorem}
\begin{proof}%[Theorem~\ref{cor:xi_alpha_xj_beta}]
Let $z_i$ be the number of zeros of $X_i$ ($z_i = N - w_i$), 
and $\overline{X}_i$ the complement of $X_i$ with $z_i$ ones and $w_i$ zeros. We proceed case by case as follows. 

In the case of \eqref{eq:00}, $(X_i, X_j) \neq (0, 0) \rightarrow (\overline{X}_i, X_j) \neq (1, 0)$; 
and, by Lemma~\ref{lem:xi_1_xj_0}, $w_j - z_i \geq 1 \rightarrow w_j - (N-w_i) \geq 1 \rightarrow w_i + w_j \geq N + 1$.

In the case of \eqref{eq:01}, we apply Lemma~\ref{lem:xi_1_xj_0} after flipping the places of $X_i$ and $X_j$.

In the case of \eqref{eq:10}, we apply Lemma~\ref{lem:xi_1_xj_0} directly.

In the case of \eqref{eq:11}, similar to the first one, 
$(X_i, X_j) \neq (1, 1) \rightarrow (X_i, \overline{X}_j) \neq (1, 0)$; 
and, by Lemma~\ref{lem:xi_1_xj_0}, $z_j - w_i \geq 1 \rightarrow (N-w_j)-w_i \geq 1 \rightarrow w_i+w_j \leq N-1$. \qed
\end{proof}

\subsection{Optimality Conditions}\label{sec:lb}

On the one hand, forbidding a pair of assignments prevents the combination of certain columns in the solution, 
such as using pairs of balanced columns. 
Because balanced columns can typically be combined in larger number (see discussion in Section~\ref{sec:shadow}),  
we may expect the number of required tests to grow. 

On the other hand, we do not need to cover a pair of assignment that has been forbidden, 
which implicitly eliminates covering constraints that we have taken for granted throughout the paper. 
Sometimes this benefit outweighs the inconvenience of having more pairwise constraints. 
To the best of our knowledge, 
a reduction on the number of tests such as from $\BFM^{(7)}$ to $\BFM^{(8)}$ have only been previously discussed in the context of  nonbinary parameters~\cite{bryce2006prioritized}.

To prove optimality given that the number of tests may decrease, 
we may rely on a smaller coverage problem on parameters not constraining one another:
\begin{lemma}[Lower Bound]\label{lem:lb}
Let $X_1, X_2, \ldots, X_{k'}$ be binary parameters of a reduced constrained pairwise covering problem with $k \geq k'$ parameters, 
for which there are no forbidden pairwise assignments between the first $k'$ parameters. 
Then any valid covering array on $k$ parameters has at least $\CAN(2,k',2)$ tests.
\end{lemma}
\begin{proof}%[Lemma~\ref{lem:lb}]
If there are no forbidden assignments between the first $k'$ parameters, 
then every pair of assignments among those should be covered. 
Hence, at least as many tests as in the unconstrained case on $k'$ parameters are needed. \qed
\end{proof}
For example, 
matrix $\BFM^{(8)}$ denotes a covering problem on $k=5$ variables 
subject to $(X_3, X_5) \neq (1, 0)$ and $(X_4, X_5) \neq (1, 0)$. 
Since there is no forbidden pair of assignments between the first four parameters, 
then at least $\CAN(2,4,2)=5$ rows are needed under those two constraints, 
and therefore $\BFM^{(8)}$ is optimal.

\subsection{Going Off Balance}\label{sec:cliques}

We have seen a few plot twists, 
but there is more yet to come. 
In this section, 
we will focus on the parameters that are subject to pairwise constraints. 
In some cases, 
those parameters may have some---or even all---columns being unbalanced in optimal covering arrays. 
In full disclosure, 
the cases discussed in this section emerged by analyzing preliminary computational experiments. 
They are likely not the only unusual cases to consider, 
but they are possibly among the most common up to the number of parameters and pairwise constraints used in our experiments in Section~\ref{sec:results}, given that we came across them a few times.

To begin our discussion, 
consider matrices $\BFM^{(10)}$ and $\BFM^{(11)}$, 
each of which having submatrices of ones in the secondary diagonal defined by the lines below:
\[
\BFM^{(10)}{=}\left( \begin{array}{cc|cc} % K_2,2
0 & 0 ~ & ~ 1 & 1 \\ 
0 & 1 ~ & ~ 1 & 1 \\ 
1 & 0 ~ & ~ 1 & 1 \\ 
\hline
1 & 1 ~ & ~ 0 & 0 \\ 
1 & 1 ~ & ~ 0 & 1 \\ 
1 & 1 ~ & ~ 1 & 0 
\end{array} \right),
\BFM^{(11)}{=}\left( \begin{array}{cc|cccc} % K_2,4
0 & 0 ~ & ~ 1 & 1 & 1 & 1 \\ 
0 & 1 ~ & ~ 1 & 1 & 1 & 1 \\ 
1 & 0 ~ & ~ 1 & 1 & 1 & 1 \\ 
\hline
1 & 1 ~ & ~ 1 & 0 & 0 & 0 \\ 
1 & 1 ~ & ~ 0 & 1 & 0 & 0 \\ 
1 & 1 ~ & ~ 0 & 0 & 1 & 0 \\
1 & 1 ~ & ~ 0 & 0 & 0 & 1
\end{array} \right).
\]
Matrix $\BFM^{(10)}$ is an optimal covering array on $k=4$ parameters subject to 
$(X_1, X_3) \neq (0, 0)$, $(X_1, X_4) \neq (0, 0)$, $(X_2, X_3) \neq (0, 0)$, and $(X_2, X_4) \neq (0, 0)$, 
with $N=6$ rows and all columns having 4 ones and 2 zeros. 
In turn, matrix $\BFM^{(11)}$ is an optimal covering array on $k=6$ parameters subject to 
$(X_1, X_3) \neq (0, 0)$, $(X_1, X_4) \neq (0, 0)$, $(X_1, X_5) \neq (0, 0)$, $(X_1, X_6) \neq (0, 0)$, $(X_2, X_3) \neq (0, 0)$, $(X_2, X_4) \neq (0, 0)$, $(X_2, X_5) \neq (0, 0)$, and $(X_2, X_6) \neq (0, 0)$, 
with $N=7$ rows and the first two columns having 5 ones and 2 zeros.
In both cases, 
we need one more test in comparison to the unconstrained case.

If we model a graph with the parameters as vertices and edges between parameters in $0$--$0$ pairwise constraints, 
the cases just described would correspond to complete bipartite graphs. 
In the first case, the graph would be isomorphic to $K_{2,2}$, with one partition corresponding to $A := \{ X_1, X_2\}$ and the other to $B := \{X_3, X_4\}$. 
In the second case, 
the graph would be isomorphic to $K_{2,4}$, 
with one partition corresponding to $A := \{X_1, X_2\}$ and the other to $B := \{X_3, X_4, X_5, X_6\}$. 
That makes it easier to observe the following: 
if the value of a parameter in either partition is set to zero, 
then the $0$--$0$ pairwise constraints imply that the value of all the parameters in the other partition should be one. 
For the submatrices of $\BFM^{(10)}$ and $\BFM^{(11)}$ defined by the vertical and horizontal lines, 
the main diagonal having zeros for the parameters in $A$ (top left) and in $B$ (bottom right) 
imply that the secondary diagonal has submatrices of ones.

To proceed, we formalize the use of graphs to represent pairwise constraints:

\begin{definition}[Induced Constraint Graph]
For a subset of parameters $P$  of a constrained pairwise covering problem, 
let $G[P] = (P, E)$ be the induced constraint graph on $P$ with vertices in $P$ corresponding to the parameters 
and edges in $E$ corresponding to all pairs of parameters in $P$ having a pairwise constraint, be that constraint explicitly defined or implied by the other constraints.
\end{definition}

The circumstances above lead to nontrivial lower bounds on the number of rows of covering arrays and on the number of ones in some of their columns:

\begin{lemma}[$0$--$0$ $K_{m,n}$ Bound]\label{lem:kmn}
If a subset of parameters $P$ of a reduced constrained pairwise covering problem has an induced constraint graph $G[P]$ 
that is isomorphic to the complete bipartite graph $K_{m,n}$  with partitions $A$ and $B$, 
 $m := |A| \geq 2$ and $n := |B| \geq 2$, 
and such that $(X_i, X_j) \neq (0, 0) ~ \forall X_i \in A, X_j \in B$, 
then 
\begin{equation}\label{eq:00_kmn_n}
N \geq \CAN(2,m,2) + \CAN(2,n,2) - 2     
\end{equation}
and 
\begin{align}%{ccc}
w_i & \geq \CAN(2,n,2) & \forall X_i \in A, \label{eq:00_kmn_wia} \\ 
w_i & \geq \CAN(2,m,2) & \forall X_i \in B,  \label{eq:00_kmn_wib} 
\end{align}
where $w_i$ is the number of ones in the column associated with parameter $X_i$.
\end{lemma}
\begin{proof}
To break the symmetry in the proof, 
let $P_1 \in \{A, B\}$ (i.e., one of the partitions among $A$ and $B$), $P_2 \in \{A, B\} \setminus \{ P_1 \}$ (i.e., the other partition), 
so that any claim proven for $P_1$ applies to $P_2$ and vice versa. 

Due to the pairwise constraints $(X_i, X_j) \neq (0,0)$ for $X_i \in P_1$ and $X_j \in P_2$, 
it follows that $X_i \in P_1 \wedge X_i = 0 \rightarrow X_j = 1 ~\forall X_j \in P_2$. 
In other words, 
any row with a zero in the columns corresponding to parameters in $P_1$ 
has a value of one in all columns corresponding to parameters in $P_2$. 

Since there are no pairwise constraints within $P_1$, 
a valid covering array must cover all pairwise assignments $(X_i, X_j) = (\alpha, \beta)$ for $X_i \in P_1$ and $X_j \in P_1$. 
But with the exception of $(\alpha, \beta) = (1, 1)$,
all other pairwise assignments within $P_1$ are covered by rows in which all parameters in $P_2$ are set to 1.

Considering the covering array problem on the parameters of $P_1$ alone, 
it takes a covering array with at least $\CAN(2,|P_1|,2)$ rows to cover all pairwise assignments. 
Moreover, no more than one of the rows of such a covering array has only the value one, or it would be a duplicate. 
Hence, it follows that a minimum of $\CAN(2,|P_1|,2) - 1$ rows of any valid covering array on all parameters of $P$ must have at least one value zero in columns corresponding to parameters in $P_1$.

Because the rows with a value zero in columns corresponding to parameters in $P_1$ and to parameters in $P_2$ are disjoint, it follows that $N \geq (\CAN(2,|P_1|,2) - 1) + (\CAN(2,|P_2|,2) - 1) = \CAN(2,m,2) + \CAN(2,n,2) - 2$, 
proving \eqref{eq:00_kmn_n}.

Finally,
the column associated with each parameter in $P_1$ in any valid covering array has at least as many ones as the number of rows in which the columns associated with the parameters in $P_2$ have at least one zero ($\CAN(2,|P_2|,2)-1$), 
since for all of those the value of all parameters in $P_1$ is one. If there are other parameters in $P_1$, there should be at least another row with the value one for each parameter $X_i \in P_1$ to cover the pairwise assignment for $(X_i, X_j) = (1,0)$ for another $X_j \in P_1$. 
Hence, $w_i \geq \CAN(2,|P_2|,2) ~\forall X_i \in P_1$, 
proving \eqref{eq:00_kmn_wia}--\eqref{eq:00_kmn_wib}. \qed
\end{proof}

The proof above entails the construction used for obtaining matrices $\BFM^{(10)}$ and $\BFM^{(11)}$. 
Let $A := \{X_1, \ldots,$ $X_m\}$ and $B := \{X_{m+1}, \ldots, X_{m+n}\}$.
Without loss of generality, 
let
$\overline{\BFM}^A \in \CAS(2,m,2)$ and $\overline{\BFM}^B \in \CAS(2,n,2)$ be valid covering arrays with rows of ones, 
since any other valid covering array can be turned to a valid covering array with a row of ones upon complementing all the columns having a value of zero in a chosen row. 
Moreover, let $\BFM^A$ and $\BFM^B$ be the submatrices of $\overline{\BFM}^A$ and $\overline{\BFM}^B$ without the row of ones. Finally, 
let
\[
\BFM^{(12)} := \left[ 
\begin{array}{cc}
\BFM^A & \BFW_{\CAN(2,m,2)-1, n} \\
\BFW_{\CAN(2,n,2)-1, m} & \BFM^B
\end{array}
\right]
\]
be a matrix containing $\BFW_{i,j}$ as submatrices of ones with $i$ rows and $j$ columns. 
Hence, matrix~$\BFM^{(12)}$ generalizes the construction of both $\BFM^{(10)}$ and $\BFM^{(11)}$.

This is the point in which we observe that \emph{optimal covering arrays may never have balanced columns for certain parameters}. 
If $|B| > |A|$ and there are possibly more rows in $\overline{\BFM}^B$ than in $\overline{\BFM}^A$, 
as is the case with matrix $\BFM^{(11)}$, 
then the columns for the parameters in $A$ can have significantly more ones than zeros.

More generally, any valid covering array on a set of parameters $P$ having a set of pairwise constraints that meet the conditions of Lemma~\ref{lem:kmn} would have that same structure upon permutation of rows and columns, 
with the exception that the covering arrays $\overline{\BFM}^A$ and $\overline{\BFM}^B$ on the subset of  parameters $A$ and $B$ are not optimal in all covering arrays.

Note, however, that there is nothing special about constraining $(0,0)$ pairs.  
Hence, we generalize Lemma~\ref{lem:kmn} to the cases that can be obtained by replacing some of the parameters with their complements if we adjust the pairwise constraints accordingly. 
We first characterize such parameter partitionings:

\begin{definition}[Coherent Partition of Induced Constraint Graph]
An induced constraint graph $G[P] = (P,E)$ has a \emph{coherent partition} of parameters $(P_0, P_1)$ if, 
for every pairwise constraint $(X_i, X_j) \neq (\alpha, \beta)$, 
it follows that $X_i \in P_{\alpha}$ and $X_j \in P_{\beta}$.
\end{definition}

In other words, 
a partitioning of parameters is coherent is the value of every parameter in the forbidden assignments associated with it is always the same (in this case, the value $0$ for parameters in $P_0$ and $1$ for parameters in $P_1$).

\begin{theorem}[Coherent $K_{m,n}$ Bound]\label{thm:kmn}
If a subset of parameters $P$ of a reduced constrained pairwise covering problem has (i) a coherent partition of parameters $(P_0, P_1)$; and (ii) an induced constraint graph $G[P]$ 
that is isomorphic to the complete bipartite graph $K_{m,n}$  with partitions $A$ and $B$, 
 $m := |A| \geq 2$ and $n := |B| \geq 2$, 
then 
\begin{equation}\label{eq:kmn_n}
N \geq \CAN(2,m,2) + \CAN(2,n,2) - 2     
\end{equation}
and 
\begin{align}%{ccc}
w_i & \geq \CAN(2,n,2) & \forall X_i \in A \cap P_0, \label{eq:kmn_wia} \\ 
w_i & \geq \CAN(2,m,2) & \forall X_i \in B \cap P_0,  \label{eq:kmn_wib} \\ 
w_i & \leq N - \CAN(2,n,2) & \forall X_i \in A \cap P_1, \label{eq:kmn_zia} \\ 
w_i & \leq N - \CAN(2,m,2) & \forall X_i \in B \cap P_1, \label{eq:kmn_zib} 
\end{align}
where $w_i$ is the number of ones in the column associated with parameter $X_i$.
\end{theorem}
\begin{proof}
By taking the complement of all parameters in $P_1$ in the forbidden assignments, 
we obtain a new constrained pairwise covering problem to which Lemma~\ref{lem:kmn} applies. 
With a similar logic to Proposition~\ref{prop:alternate_complement}, 
except that considering the constrained case, 
any valid covering array to this new covering problem and to the covering problem in the statement that we are proving are interchangeable upon complement of the parameters that are in $P_1$ for the latter.

Due to the interchangeability of valid covering arrays between the two covering problems, the bound \eqref{eq:00_kmn_n} from Lemma~\ref{lem:kmn} implies the claimed bound \eqref{eq:kmn_n}.

For the parameters in $P_0$, for which we do \emph{not} take complements between the valid covering arrays of the two covering problems, 
the bounds \eqref{eq:00_kmn_wia} and \eqref{eq:00_kmn_wib} from Lemma~\ref{lem:kmn} directly imply the claimed bounds \eqref{eq:kmn_wia} and \eqref{eq:kmn_wib}.

For the parameters in $P_1$, for which we \emph{do} take complements between the valid covering arrays of the two covering problems, 
the bounds \eqref{eq:00_kmn_wia} and \eqref{eq:00_kmn_wib} from Lemma~\ref{lem:kmn} actually imply different bounds on columns of the covering problem:
\begin{align}%{ccc}
z_i & \geq N - \CAN(2,n,2) & \forall X_i \in A \cap P_1, \label{eq:kmn_zia_proof} \\ 
z_i & \geq N - \CAN(2,m,2) & \forall X_i \in B \cap P_1, \label{eq:kmn_zib_proof} 
\end{align}
where $z_i$ is the number of ones in the column associated with parameter $X_i$.
With $z_i = N - w_i$, 
those new bounds \eqref{eq:kmn_zia_proof} and \eqref{eq:kmn_zib_proof}  
imply the claimed bounds \eqref{eq:kmn_zia} and \eqref{eq:kmn_zib}, respectively.
\qed
\end{proof}

The use of graphs for modeling pairwise constraints and their complements has been first proposed in Danziger et al.~\cite{danziger2009cafe},  
where an algorithm of exponential complexity produces covering arrays of minimum size.  
Such graphical models have also been used in complexity results~\cite{danziger2009cafe,maltais2010cafe} and  
applied to study special lower bounds for this problem~\cite{danziger2009cafe,yang2021decomposing}. 
In fact, 
the next results are also inspired by 
the use of a clique of pairwise constraints in Danziger et al.'s algorithm. 

To illustrate this next case, let us first consider matrices $\BFM^{(13)}$ and $\BFM^{(14)}$:
\[
\BFM^{(13)}{=} \left( \begin{array}{ccc} % K_3
0 & 1 & 1 \\
1 & 0 & 1 \\
1 & 1 & 0 
\end{array} \right),~
\BFM^{(14)}{=} \left( \begin{array}{ccc|c} % K_3
0 & 1 & 1 ~ & ~ 1  \\
0 & 1 & 1 ~ & ~ 0  \\ \hline 
1 & 0 & 1 ~ & ~ 1  \\
1 & 0 & 1 ~ & ~ 0  \\ \hline
1 & 1 & 0 ~ & ~ 1  \\
1 & 1 & 0 ~ & ~ 0  \\ 
\end{array} \right).
\]
Both of those are optimal covering arrays subject to the same pairwise constraints: 
$(X_1, X_2) \neq (0, 0)$, $(X_1, X_3) \neq (0, 0)$, and $(X_2, X_3) \neq (0, 0)$.
In other words, we have a clique of constraints among the first three parameters while all other parameters are unconstrained. 
When one of those first three parameters has the value zero, the other two parameters must have the value one. 

What changes from one case to another is the existence and an unconstrained parameter. 
We separate the constrained parameters from the unconstrained parameter with a vertical line. 
In the first case with $k=3$ parameters, given that there are no unconstrained parameters, 
we actually have \emph{one less test} than in the optimal unconstrained case. 
Once we add one unconstrained parameter and $k=4$, however, 
the need to cover both $(X_i, X_4) = (0, 0)$ and $(X_i, X_4) = (0, 1)$ for every constrained parameter $X_i$ 
implies that we need two rows in which each constrained parameter has the value zero. 
Given the clique of pairwise constraints, 
that effectively doubles the number of rows in comparison to $k=3$. 
Now that we have rows that are identical with respect to the constrained parameters alone, 
we separate rows that are distinct on the constrained parameters with horizontal lines. 
With $k=4$ parameters, we have \emph{one more test} than in the optimal unconstrained case.

This is point in which we observe that \emph{all parameters may have unbalanced columns in optimal covering arrays}. 
If we consider a covering array on $k \geq 3$ parameters with constraints of the form $(X_i, X_j) \neq (0,0) ~ \forall i \in \{1, \ldots, k\}, j \in \{1, \ldots, k\} \setminus \{ i \}$, 
hence forming a fully connected graph on $k$ parameters and generalizing the case with matrix~$\BFM^{(13)}$, 
then each column has $k-1$ ones and one zero.
However, a fully connected graph is a rather specific and unlikely case.

Hence, our object of interest now is not exactly a clique, 
but a combination of a clique and isolated vertices. 
That has been denoted as a \emph{sparse split graph} in the literature~\cite{SparseSplit}.
We use the following characterization for convenience:

\begin{definition}[Sparse Split Induced Constraint Graph]
An induced constraint graph $G[P] = (P,E)$ is \emph{sparse split} with a partition of parameters $(P_C, P_U)$ if the induced constraint graph $G[P_C]$ is isomorphic to the clique $K_n$, where $n := |P_C|$, and there are no pairwise constraints between the parameters in $P_U$ and the parameters in $P$.
\end{definition}

However, matrices $\BFM^{(13)}$ and $\BFM^{(14)}$ are not sufficient to capture all the nuance of what happens when the pairwise constraints define a sparse split graph. 
To continue the discussion, 
let us consider matrices $\BFM^{(15)}$ and $\BFM^{(16)}$:
\[
\BFM^{(15)}{=} \left( \begin{array}{ccc|cccc} % K_3
0 & 1 & 1 ~ & ~ 1 & 1 & 1 & 1 \\
0 & 1 & 1 ~ & ~ 0 & 0 & 0 & 0  \\ \hline 
1 & 0 & 1 ~ & ~ 1 & 1 & 0 & 0  \\
1 & 0 & 1 ~ & ~ 0 & 0 & 1 & 1  \\ \hline
1 & 1 & 0 ~ & ~ 1 & 0 & 1 & 0  \\
1 & 1 & 0 ~ & ~ 0 & 1 & 0 & 1  \\ 
\end{array} \right),~
\BFM^{(16)}{=} \left( \begin{array}{ccccccc||cc} % K_3
0 & 1 & 1 & 1 & 1 & 1 & 1 ~ & ~ 1 & 1 \\
0 & 1 & 1 & 0 & 0 & 0 & 0 ~ & ~ 0 & 0 \\  
1 & 0 & 1 & 1 & 1 & 0 & 0 ~ & ~ 1 & 0 \\
1 & 0 & 1 & 0 & 0 & 1 & 1 ~ & ~ 0 & 1 \\ 
1 & 1 & 0 & 1 & 0 & 1 & 0 ~ & ~ 0 & 0 \\
1 & 1 & 0 & 0 & 1 & 0 & 1 ~ & ~ 0 & 0 \\ \hline \hline 
1 & 1 & 0 & 0 & 0 & 0 & 0 ~ & ~ 1 & 1
\end{array} \right).
\]
Both of those are optimal covering arrays subject to the same pairwise constraints mentioned before. 
In the first case with $k=7$ parameters, 
we keep a vertical line separating constrained parameters from unconstrained parameters, 
as well as horizontal lines separating distinct rows on the constrained parameters. 
When we compare the cases of $k=4$ (matrix $\BFM^{(14)}$) and $k=7$ (matrix $\BFM^{(15)}$), 
we see that adding more unconstrained parameters did not impact the optimal number of tests required 
(in fact, the number of tests for $k=7$ is the \emph{same} as in the optimal unconstrained case). 
In the second case with $k=9$ parameters, 
however, 
there is one more test than before, 
and therefore \emph{one more test} than in the optimal unconstrained case. 
We use the double vertical and horizontal lines to highlight that matrix $\BFM^{(15)}$ is the submatrix of $\BFM^{(16)}$ on the first 6 rows and 7 columns. 
We establish with the next result the need for at least 2 rows per parameter in the clique, 
and possibly more rows in cases such as with $\BFM^{(16)}$:

\begin{lemma}[$0$--$0$ $K_{n}$ Bound\footnote{We will keep referring to this family of bounds as $K_n$ because the clique is central to the result and cliques are more commonly known than sparse split graphs.}]\label{lem:kn}
If a subset of parameters $P$ of a reduced constrained pairwise covering problem induces a sparse split constraint graph $G[P]$ with partitions $(P_C, P_U)$,  
$n := |P_C| \geq 3$ 
and $m := |P_U| \geq 1$,  
then
\begin{equation}\label{eq:00_kn_n}
N \geq 2 n + \alpha,
\end{equation}
where $2n+\alpha$, $\alpha \geq 0$, is the smallest number of rows such that 
(a) there are $m$ columns defining a valid covering array on the parameters in $P_U$; and 
(b) we can partition that covering array in $n$ subsets of tests, 
each of which having $X_i = 0$ and $X_i = 1$ for each parameter $X_i \in P_U$.  
Moreover, 
\begin{align}
w_i \geq 2 (n-1) & ~~~ \forall i \in P_C, \label{eq:00_kn_wi} \\
\sum_{i \in P_C} w_i \geq (2n + \alpha) (n-1), \label{eq:00_kn_swi}
\end{align}
where $w_i$ is the number of ones in the column of $X_i \in P_C$ in any solution.
\end{lemma}
\begin{proof}
Let us assume that $P_C := \{X_1, \ldots,$ $X_n\}$. 
Hence, let $M := [ M_C ~ | ~ M_U ]$ be any valid covering array for the constrained pairwise covering problem on $P$, 
in which $M_C$ has $n$ columns and $M_U$ has $m$ columns. 
Without loss of generality, we may assume that the first rows of $M_C$ have the value zero in $X_1$, 
then the next rows of $M_C$ have the value zero in $X_2$, 
and so on, 
since each test has at most one value zero in $P_C$, 
and thus we can rearrange the rows of any covering array $M$ in such a way. 
Since there are no pairwise constraints between parameters in $P_C$ and $P_U$, 
there should be at least two rows with $X_i = 0$ for each $X_i \in P_C$ 
to cover the pairwise assignments $(X_i, X_j) = (0, 0)$ and $(X_i, X_j) = (0, 1)$ for each $X_j \in P_U$. 
With $n$ of such parameters, it follows that $N \geq 2n$, 
proving \eqref{eq:00_kn_n}.

Since there are no pairwise constraints within the parameters in $P_U$,
it follows that $M_U \in \CA(N;2,m,2)$, 
hence proving (a).
Moreover, $M_U$ is a valid covering array on $P_U$ in which 
$X_j = 0$ and $X_j = 1$ for every $X_j \in P_U$ in the subset of tests for which $X_i = 0$ for each $X_i \in P_C$, 
since otherwise we would not cover $(X_i, X_j) = (0, 0)$ or $(X_i, X_j) = (0, 1)$. 
By assumption, 
those subsets of tests correspond to contiguous subsets of rows of $M$, 
the first with $X_1 = 0$,
the second with $X_2 = 0$, 
and so on, 
with a total of $n$ subsets and therefore proving (b).

Since each row has no more than one zero in the first $n$ columns, and thus at least $n-1$ ones, 
then $\sum_{i \in P_C} w_i \geq (2n+\alpha)(n-1)$, 
proving~\eqref{eq:00_kn_wi}.
Since each of the first $n$ columns has $n-1$ sets with at least 2 rows each having value one in the column, 
then $w_i \geq 2 (n-1) ~\forall X_i \in P$, 
proving~\eqref{eq:00_kn_swi}. \qed
\end{proof}

Similarly to what we did before with Lemma~\ref{lem:kmn}, 
we now use the concept of a coherent partition of parameters to generalize the result from Lemma~\ref{lem:kn}:

\begin{theorem}[Coherent $K_{n}$ Bound]\label{thm:kn}
If a subset of parameters $P$ of a reduced constrained pairwise covering problem (i) has a coherent partition of parameters $(P_0, P_1)$; and (ii) induces a sparse split constraint graph $G[P]$ with partitions $(P_C, P_U)$,  
$n := |P_C| \geq 3$ 
and $m := |P_U| \geq 1$,  
then
\begin{equation}\label{eq:kn_n}
N \geq 2 n + \alpha,
\end{equation}
where $2n+\alpha$, $\alpha \geq 0$, is the smallest number of rows such that 
(a) there are $m$ columns defining a valid covering array on the parameters in $P_U$; and 
(b) we can partition that covering array in $n$ subsets of tests, 
each of which having $X_i = 0$ and $X_i = 1$ for each parameter $X_i \in P_U$.  
Moreover, 
\begin{align}
w_i \geq 2 (n-1) & ~~~ \forall i \in P_C \cap P_0, \label{eq:kn_wi} \\
w_i \leq N - 2 (n-1) & ~~~ \forall i \in P_C \cap P_1, \label{eq:kn_win} \\
\sum_{i \in P_C \cap P_0} w_i + \sum_{i \in P_C \cap P_1} (N-w_i) \geq (2n + \alpha) (n-1), \label{eq:kn_swi}
\end{align}
where $w_i$ is the number of ones in the column of $X_i \in P_C$ in any solution.
\end{theorem}
\begin{proof}
This proof is very similar in structure to the proof of Theorem~\ref{thm:kmn}.

By taking the complement of all parameters in $P_1$ in the forbidden assignments, 
we obtain a new constrained pairwise covering problem to which Lemma~\ref{lem:kn} applies. 
With a similar logic to Proposition~\ref{prop:alternate_complement}, 
except that considering the constrained case, 
any valid covering array to this new covering problem and to the covering problem in the statement that we are proving are interchangeable upon complement of the parameters that are in $P_1$ for the latter.

Due to the interchangeability of valid covering arrays between the two covering problems, the bound \eqref{eq:00_kn_n} from Lemma~\ref{lem:kn} implies the claimed bound \eqref{eq:kn_n}. 
Similarly, 
the validity of conditions (a) and (b) from Lemma~\ref{lem:kn} imply the validity of the claimed conditions (a) and (b) here. 

For the parameters in $P_0$, for which we do \emph{not} take complements between the valid covering arrays of the two covering problems, 
the bound \eqref{eq:00_kn_wi} from Lemma~\ref{lem:kn} directly implies the claimed bound \eqref{eq:kn_wi}.

For the parameters in $P_1$, for which we \emph{do} take complements between the valid covering arrays of the two covering problems, 
the bounds \eqref{eq:00_kn_wi} and \eqref{eq:00_kn_swi} from Lemma~\ref{lem:kn} actually imply different bounds on columns of the covering problem:
\begin{align}%{ccc}
z_i \leq N - 2 (n-1) & ~~~ \forall i \in P_C \cap P_1, \label{eq:kn_win_proof} \\
\sum_{i \in P_C \cap P_0} w_i + \sum_{i \in P_C \cap P_1} z_i \geq (2n + \alpha) (n-1), \label{eq:kn_swi_proof}
\end{align}
where $z_i$ is the number of ones in the column associated with parameter $X_i$.
With $z_i = N - w_i$, 
those new bounds \eqref{eq:kn_win_proof} and \eqref{eq:kn_swi_proof} 
imply the claimed bounds \eqref{eq:kn_win} and \eqref{eq:kn_swi}, respectively.
\qed
\end{proof}

To finalize this part, 
we characterize a relevant case in which $\alpha \neq 0$. 
This is motivated by the construction of matrix $\BFM^{(15)}$, 
which we repeat for convenience:
\[
\BFM^{(15)}{=} \left( \begin{array}{ccc|cccc} % K_3
0 & 1 & 1 ~ & ~ 1 & 1 & 1 & 1 \\
0 & 1 & 1 ~ & ~ 0 & 0 & 0 & 0  \\ \hline 
1 & 0 & 1 ~ & ~ 1 & 1 & 0 & 0  \\
1 & 0 & 1 ~ & ~ 0 & 0 & 1 & 1  \\ \hline
1 & 1 & 0 ~ & ~ 1 & 0 & 1 & 0  \\
1 & 1 & 0 ~ & ~ 0 & 1 & 0 & 1  \\ 
\end{array} \right).
\]
For the unconstrained parameters (to the right of the vertical line), 
note how there is a zero and a one for each column within every group defined by the horizontal lines. 
Essentially, our only degree of freedom is in deciding if the zero or the one goes first in each column within each group. 
That naturally produces balanced columns for the unconstrained parameters if each group has two rows. 
By restricting the columns within the first group to always have the value one first, 
we prevent having columns that are complements. 
By changing the order of the rows having each value in the other two groups, 
we can only obtain the four distinct columns that we observe to the right of the vertical line. 
Hence, 
having one more parameter would entail having another row in any valid covering array. 
We generalize the relationship between the number of constrained and unconstrained parameters and the need for an additional test as follows:

\begin{corollary}[Extra Test Threshold for Coherent $K_n$ Bound]
    When the conditions of Theorem~\ref{thm:kn} apply, 
    with $n$ as the number of constrained parameters inducing a clique and $m$ as the number of unconstrained parameters, 
    then
    \[
    m > 2(n-1) \rightarrow \alpha \geq 1.
    \]
\end{corollary}
\begin{proof}
Without loss of generality, 
let us assume that $P_1 = \emptyset$; 
since we can define an equivalent covering problem by taking the complement of all parameters in $P_1$ in the forbidden assignments, 
and then obtain a valid covering array for the original covering problem by taking the complement of the columns associated with $P_1$ in a valid covering array for the new problem, 
as in previous proofs.

Now let us suppose, for contradiction, that $\alpha = 0$ yields a tight bound on $N$. 
In other words, there is a valid covering array with $2n$ tests. 
In such an array, there are two rows in which $X_i = 0$ for each $X_i \in P_C$, 
which are disjoint from the rows in which $X_j = 0$ for a different parameter $X_j \in P_C$.

Since this array should cover the pairwise assignments $(X_i, X_j) = (0, 0)$ and $(X_i, X_j) = (0, 1)$ between each $X_i \in P_C$ and each unconstrained parameter $X_j \in P_U$, 
then there is one row covering each such assignment.
If two columns on the unconstrained parameters $X_j \in P_U$ cover all of those assignments with the same rows, 
then those columns are identical, 
and that would contradict Proposition~\ref{prop:unique_cols} with respect to the columns on the unconstrained parameters in $P_U$ defining a valid covering array. 
For each pair of columns on $P_U$, the rows containing the assignments $X_j = 0$ and $X_j = 1$ should thus be different within the group in which $X_i = 0$ for at least one parameter $X_i \in P_C$. 
There are $2n$ columns that can be defined by alternating the order of the assignments to $X_j$ within those groups. 
But since half of those would be complements since there are only two rows per group, 
then no more than $2(n-1)$ of such columns can define a valid covering array on the unconstrained parameters in $P_U$ due to Proposition~\ref{prop:not_complement}.
However, that would contradict that $m > 2(n-1)$, 
hence implying that the initial assumption is false and thus that $\alpha \geq 1$ in this case. 
\end{proof}

\subsection{Regaining Balance}\label{sec:conjecture}

Let us have a moment of hope now. 
In this section, 
we will focus on the parameters that are not subject to pairwise constraints.
Despite the lengthy discussion about optimal covering arrays necessarily having unbalanced columns, 
that has so far applied only to the parameters that are subject to pairwise constraints.
Hence,
we have not yet seen a contradiction to the use of balanced columns for all other parameters. 
For that reason, we posit the following:

\begin{conj}[Balanced Columns for Unconstrained Parameters]\label{conj:balanced}
In a reduced pairwise constrained covering problem, 
there is an optimal covering array in which parameters without pairwise constraints have balanced columns. 
\end{conj}

In fact, this argument revisits the discussion from Section~\ref{sec:shadow} about using unbalanced columns in optimal unconstrained covering arrays. 
While it is possible to use unbalanced columns in some cases, 
their use may further limit the number of columns that can be put together to produce a valid covering array.

Moreover, 
if two parameters are unconstrained, 
then their columns remain interchangeable. 
Hence, 
we can significantly reduce the symmetry when formulating the problem of finding optimal covering arrays 
if we assume that we simply need a certain number of balanced columns to be assigned to the unconstrained parameters. 
Since any association of those columns to parameters is possible, 
then the exact association is irrelevant. 
That is how Conjecture~\ref{conj:balanced} will help in the next section. 

\section{Integer Programming Models}\label{sec:milp}

Following the developments from the last two sections, 
we are now able to propose an algorithm and an heuristic for producing optimal, or at least valid and hopefully good, covering arrays. 
Our general approach is outlined in Algorithm~\ref{alg:approach}. 
It is based on solving a series of subproblems with the support of IP models, 
each considering covering arrays with a fixed number of tests. 

The elements below are common throughout our approach and IP models:
\begin{itemize}
\item the set of parameter indices $\mathcal{P} := \{ 1, \ldots, k \}$; 
\item the set of pairs of parameter indices that are subject to a pairwise constraint $\mathcal{F} \subseteq \{ (i_1,i_2) : i_1 \in \mathcal{P}, i_2 \in \mathcal{P} \setminus \{ i_1 \} \}$; and 
\item the forbidden assignment for each such pair $\mathscr{F} : \mathcal{F} \rightarrow \{0, 1\} \times \{0, 1\}$. 
\end{itemize}

\begin{algorithm}[b!]
\begin{algorithmic}[1]
{\footnotesize
\Function{Solve}{$k, \mathcal{P}, \mathcal{F}, \mathscr{F}$}
\State $N \gets \CAN(2,k,2)$ \Comment{Start with number of tests from unconstrained case}
\State $M \gets$ \Call{FindCoveringArray}{$k, \mathcal{P}, \mathcal{F}, \mathscr{F}, N$} \Comment{Try to find a covering array}
\If{$M = \emptyset$} \Comment{We assume $M = \emptyset$ if we cannot find a covering array}
\Repeat \Comment{If we do \underline{not} find a covering array at first...}
\State $N \gets N + 1$ \Comment{we look for one with \underline{more} tests}
\State $M' \gets$ \Call{FindCoveringArray}{$k, \mathcal{P}, \mathcal{F}, \mathscr{F}, N$}
\Until{$M' \neq \emptyset$} \Comment{Stop after finding the first covering array}
\State $M \gets M'$ \Comment{Use this first covering array}
\Else
\Repeat \Comment{If we \underline{do} find a covering array at first...}
\State $N \gets N - 1$ \Comment{we look for others with \underline{fewer} tests}
\State $M' \gets$ \Call{FindCoveringArray}{$k, \mathcal{P}, \mathcal{F}, \mathscr{F}, N$} 
\If{$M' \neq \emptyset$} \Comment{If a smaller covering array is found...} 
\State $M \gets M'$  \Comment{we use that covering array instead}
\EndIf
\Until{$M' = \emptyset$} \Comment{Stop when we cannot find covering arrays anymore}
\EndIf
\State \Return $M$
\EndFunction
}
\end{algorithmic}
\caption{General approach for solving reduced pairwise constrained covering array problems, starting with the optimal number of unconstrained tests}\label{alg:approach}
\end{algorithm}

Starting with as many tests $N$ as we would need in the unconstrained case, 
Algorithm~\ref{alg:approach} either 
(i) increments $N$ as far as needed while trying to find a first valid covering array; or 
(ii) decrements $N$ until we can no longer find a valid covering array. 
In addition, 
we may stop trying smaller number of tests when we reach a proven lower bound, such as in Lemma~\ref{lem:lb} and Theorems~\ref{thm:kmn} and \ref{thm:kn}. 
It may seem unusual to increase or decrease the number of tests by 1 between steps rather than, for example, perform a binary search. 
However, we typically only need to try a few different number of tests before concluding the search. 

Hence, the only part that changes between our exact and heuristic approaches is 
the implementation of function~\textproc{FindCoveringArray}. 
This function must find a valid covering array for a given number of tests $N$. 
While our ultimate goal is to minimize the number of tests, 
allowing the number of tests to vary in the IP models would make them larger and more difficult to solve.  
Hence, we use the objective function of the IP models to heuristically direct the search in consistency with the traits that we have observed and conjectured to be common in most of the optimal covering arrays.

Now we consider what the function~\textproc{FindCoveringArray} does in each case.
We first present a simplified model coupled with backtracking for fast but heuristic results (Section~\ref{sec:p2}). 
Then we present an exact model based on our theoretical results to determine if a valid covering array on $N$ tests exists~(Section~\ref{sec:mip0}) and discuss how it changes if we assume Conjecture~\ref{conj:balanced} (Section~\ref{sec:mip1}).
We also present a baseline model that does not depend on our theoretical results (Section~\ref{sec:big_mip}).

\subsection{Heuristic Approach $\text{H}$}\label{sec:p2}

We will first describe the heuristic approach, 
since it represents a simplified version of exact model presented afterwards.

\begin{subequations}
\begin{align}
    \min ~ & k N \Delta + \sum_{i \in \mathcal{P}} \delta_i \label{lin:p2_obj} \\
        \text{s.t.} ~~ & \delta_i \geq w_i - \lfloor N/2 \rfloor & \forall i \in \mathcal{P} \label{lin:p2_delta1} \\
                        & \delta_i \geq \lfloor N/2 \rfloor - w_i & \forall i \in \mathcal{P} \label{lin:p2_delta2} \\
                        & \Delta \geq \delta_i & \forall i \in \mathcal{P} \label{lin:p2_delta} \\
                        & w_{i_1} + w_{i_2} \geq N+1 & ~~ \forall (i_1,i_2) \in \mathcal{F} : \mathscr{F}(i_1,i_2) = (0,0) \label{lin:p2_unb1} \\
                        & w_{i_1} - w_{i_2} \geq 1 & ~~ \forall (i_1,i_2) \in \mathcal{F} : \mathscr{F}(i_1,i_2) = (0,1) \\
                        & w_{i_2} - w_{i_1} \geq 1 & ~~ \forall (i_1,i_2) \in \mathcal{F} : \mathscr{F}(i_1,i_2) = (1,0) \\
                        & w_{i_1} + w_{i_2} \leq N-1 & ~~ \forall (i_1,i_2) \in \mathcal{F} : \mathscr{F}(i_1,i_2) = (1,1) \label{lin:p2_unb2} \\
                        & w_i \leq N - 1 & \forall i \in \mathcal{P} \label{lin:p2_wub} \\
                        & w_i \in \mathbb{Z}^+ & \forall i \in \mathcal{P} \\
                        & \delta_i \in \mathbb{R}^+ & \forall i \in \mathcal{P} \\
                        & \Delta \in \mathbb{R}^+ \label{lin:p2_last}
\end{align}
\end{subequations}
In the IP model \eqref{lin:p2_obj}--\eqref{lin:p2_last}, 
we only decide the number of ones $w_i$ in the column used for each parameter $X_i, i \in \mathcal{P}$, 
leaving the actual choice of the column to postprocessing. 
We capture with $\delta_i$ how far each parameter $X_i$ is from having a balanced column in (\ref{lin:p2_delta1}--\ref{lin:p2_delta2}), 
and with $\Delta$ the largest of such values in~\eqref{lin:p2_delta}. 
We use those constraints to define lower bounds on $\delta_i$ and $\Delta$, 
and then we use both in the objective function~\eqref{lin:p2_obj}. 
The objective function in this case is intended to make it easier to find a feasible solution when assigning columns to each parameter afterwards. 
First, we prioritize minimizing $\Delta$ first by effectively using $kN$ as a sufficiently large constant. 
Second, we minimize the sum of the $\delta_i$ variables. 
We also ensure that the main decision variables satisfy Theorem~\ref{cor:xi_alpha_xj_beta} in (\ref{lin:p2_unb1}--\ref{lin:p2_unb2}). 
Moreover, we bound the number of ones to $N-1$ due to Assumption~\ref{ass:first} in~\eqref{lin:p2_wub}. 
Finally, we may also derive other inequalities on variables $w_i$ directly from Theorems~\ref{thm:kmn} and \ref{thm:kn} where applicable, 
but we omit those from the formulation for clarity.

Based on an optimal solution of the model above, 
we first tried to find columns with the prescribed number of ones 
by implementing a backtracking algorithm following best practices~\cite{vanbeek2006bt}. 
By assigning columns to the parameters with fewer options first, 
we were able to quickly obtain solutions to some instances---many of which backtracking-free and matching the optimal value from the exact approach.  
However, 
there were also instances in which the exact methods were successful while this approach timed out. 
In those, 
the solutions obtained by the exact methods had a different number of ones in the columns. 

Hence,
we engineered a more nuanced heuristic to exploit the structural properties discussed in the paper. 
Our goal was twofold: 
(i) for the instances that the exact methods could solve to optimality, we wanted to quickly produce solutions of similar quality; and
(ii) for the other instances, we wanted to be competitive with the most popular heuristics in use. 
We briefly outline this approach below. 

First, we only use the IP model \eqref{lin:p2_obj}--\eqref{lin:p2_last} 
to determine if a covering array of size $N$ exists. 
If not, we increment the value of $N$ and repeat. 

Second, 
we generate a few different solutions on the $w$ variables, 
which we denote as \emph{profiles}. 
Each profile defines a smaller search space to explore in order to find a feasible assignment of columns to the parameters, 
where the number of ones on the column associated with each parameter is fixed. 
Before generating the profiles, 
we use Theorems~\ref{cor:xi_alpha_xj_beta}, \ref{thm:kmn}, and \ref{thm:kn} for tightening the domains of the $w$ variables. 
We generate the first profile by trying to assign each variable in $w$ to be as closed to balanced as possible, 
and then looping over the constraints entailed by Theorem~\ref{cor:xi_alpha_xj_beta} to correct violations by adjusting in one unit the value of the variable with largest domain. 
To obtain the first profile, 
this loop is repeated until there are no violations left. 
This gives us the first profile $w = \bar{w}$.

Then we derive additional profiles by adjusting the number of ones associated with the parameters with the most  constraints on $w$ entailed by Theorem~\ref{cor:xi_alpha_xj_beta}. 
Always starting from the first profile $\bar{w}$, 
we generate separate profiles by adjusting both up and down the value in $w$ for each of the $C_1$ most constrained parameters, 
as well as with all combinations of adjusting up and down the value in $w$ for each pair among the $C_2$ most constrained parameters.  
In each of those $2 C_1 + 4 C_2$ new profiles, 
adjusting a value in $w$ may entail changes in the value of $w$ for other parameters due to Theorem~\ref{cor:xi_alpha_xj_beta} constraints, as well as to subsequent corrections similar to those used for generating the first profile. 
We keep the top $P$ resulting profiles after ranking them first by the largest slack with respect to the constraints defined by Theorem~\ref{cor:xi_alpha_xj_beta} (looser is better);  
second by the number of such constraints with zero slack (less is better); 
third by the sum of all slacks (larger is better); 
and then by how closed to balanced the values in $w$ are.

Third, we search for a covering array of size $N$. 
We use the profiles above to search for an assignment of columns to the parameters by backtracking. 
We first search over the columns having $w_i$ ones for each parameter $i$ that is entailed by the top ranked profile $w$, 
and then we search over the columns with a number of ones in $\{w_i - 1, w_i, w_i + 1\}$ for each parameter $i$ in each profile $w$ that was kept. 
We try assigning columns to the constrained parameter with fewer options left first, 
and among the columns available we prioritize those that would keep more balanced columns available after propagation. 
While searching for such assignments, 
we ensure that there are enough balanced columns left for the unconstrained parameters. 
If we exhaust the profiles, 
we search over the entire space. 
If we do not succeed in $T_N$ seconds, 
we increment $N$ and repeat. 

Finally, in case we find a covering array $\BFM$, say of size $N$, 
we may use it to find a covering array of size $N-1$ in two ways. 
The first is by generating profiles corresponding to removing each of the rows of $\BFM$, 
ranking them from most to least balanced, 
and then searching by backtracking on the first $R_1$ of them as above. 
The second is by removing each of the rows in order until obtaining $R_2$ distinct profiles, 
and then doing a backtrack-based search in which the column associated with each parameter cannot change by more than two elements. 
We limit the duration of each individual search to $T_F$ seconds. 

\subsection{Exact Approach $\Mzero$}\label{sec:mip0}

We will now describe the exact approach, 
which modifies and extends the IP model previously described to decide on its own if a valid covering array can be found or not, 
hence without nontrivial postprocessing in comparison to before.

To make the next model easier to understand and maintain consistency with prior statements and models, 
we reserve the use of the following indices and single-letter variables to specific purposes, 
to the extent that it is possible:
\begin{itemize}
    \item $\alpha, \beta$ for the values zero and one associated with parameters; 
    \item $c$ for columns;
    \item $i, i_1, i_2$ for parameter indices; 
    \item $j, j_1, j_2$ for the number of ones in columns;
    \item $k$ for the number of parameters;
    \item $\ell$ for other occasional summations; and
    \item $N$ for number of rows.
\end{itemize}

The additional elements below are used in this IP model:
\begin{itemize}
    \item the set of unconstrained parameters $\tilde{\mathcal{P}} \subseteq \mathcal{P}$; 
    \item the set of pairs of parameters without a pairwise constraint $\overline{\mathcal{F}} \subseteq \mathcal{P} \times \mathcal{P}$; 
    \item the set of allowable number of ones\footnote{This preprocessing is not necessary, but can reduce domains considerably. We find those sets based on which parameters must have more or less ones in their columns than other parameters, 
    leading to stricter bounds on their number of ones given $N$.} $\mathcal{O}_i \subseteq \{1, \ldots, N-1\}$ in a column that can be assigned to parameter $X_i, i \in \mathcal{P}$; 
    \item the set of all binary columns $\mathcal{C} := \{0,1\}^k$, and some of its subsets:
    \begin{itemize}
        \item the columns with $j$ ones $\mathcal{C}^j  := \{ c \in \mathcal{C} : \sum_{\ell=1}^k c_\ell = j\}$, 
        \item the canonical balanced columns with the value one in the first row $\mathcal{B} := \{ c \in \mathcal{C}^{\lfloor N/2 \rfloor} : c_1 = 1 \}$, and
        \item the allowable columns\footnote{Also obtained by preprocessing to reduce domains.} $\mathcal{C}_i := \{ c : c \in \mathcal{C}^j, j \in \mathcal{O}_i \}$ for parameter $X_i, i \in \mathcal{P}$; 
    \end{itemize}
    \item the indices of parameters\footnote{This is the converse of the previous preprocessing producing $\mathcal{C}_i$.} $\mathcal{P}_c \subset \mathcal{P}$ that are compatible with column $c \in \mathcal{C}$;
    \item the set of incompatible pairs of columns $\mathcal{I} \subset \mathcal{C}^2$ because at least one pairwise assignment is not covered; and
    \item the set of incompatible pairs of columns $\mathcal{I}^{(\alpha,\beta)} \subset \mathcal{C}^2$ because at least one pairwise assignment other than $(\alpha, \beta)$ is not covered. 
\end{itemize}

\ifpreprint
    % Nada
\else
    \MzeroMILP
\fi
In the IP model~\eqref{lin:objp3}--\eqref{lin:lastp3}, 
our main decision is what column $c \in \mathcal{C}$ is uniquely assigned to each parameter $X_i, i \in \mathcal{P}$ 
through the binary variable $V_{i c}$. 
Each parameter is assigned a column by~\eqref{lin:vpc1}, 
and each column is assigned to at most one parameter by~\eqref{lin:vpc2}.

Instead of a single variable $w_i$ for the number of ones in the column associated with $X_i$ as before, 
we use the binary variable $W_{i j}$ to more easily capture each possible number of ones $j$ for the column assigned to parameter $X_i$. 
Each parameter has a number of ones for the column assigned to it by~\eqref{lin:wpi1}, 
and the columns can only be assigned if they have that corresponding number of ones due to~\eqref{lin:wpi2}. 
Moreover, those variables make it easier to capture the inequalities from Theorem~\ref{cor:xi_alpha_xj_beta} for each possible number of ones with (\ref{lin:unb1p3}--\ref{lin:unb2p3}). 

For each pair of parameters, 
we define inequalities avoiding the assignment of incompatible columns when there is a pairwise constraint between them with~\eqref{lin:vv1} and when there is not with~\eqref{lin:vv2}, 
since the columns that would cover all and only the allowable pairwise assignments in each case would be different. 

To exploit the fact that canonical balanced columns are all mutually compatible, 
we use an additional decision variable $U_c$ denoting if a balanced column is used for an unconstrained parameter $X_i$. 
That way, we can subject $V_{i c} = 1$ to when $U_c = 1$ by~\eqref{lin:uuse}; 
and then limit the use of such a column, regardless of which parameter uses it, if a conflicting column is assigned to a constrained parameter with~\eqref{lin:uconflict}. 
Please note that this is not yet exploiting Conjecture~\ref{conj:balanced}, 
since an unbalanced column may still be assigned to an unconstrained parameter. 
However, that makes that case easier to discuss in Section~\ref{sec:mip1}.

Finally, the objective function~\eqref{lin:objp3} is intended to guide the solver toward a feasible solution: 
we aim to maximize the number of canonical balanced columns used, 
since they are automatically all compatible with one another, 
by penalizing any other choice of column for a parameter and even more so if the assigned column has a greater absolute difference between the number of zeros and ones.  

\ifpreprint
    \MzeroMILP
\fi

As before, 
we can also derive valid inequalities from Theorems~\ref{thm:kmn} and \ref{thm:kn} to strengthen the formulation. 
They are not strictly necessary, but they certainly make it easier to find a valid covering array when the conditions of those theorems hold. 
Hence, 
we also omit them from the formulation for clarity.

In practice, 
we can stop solving the model when a feasible solution is found, 
since the goal of each model is to find a feasible solution for a given $N$. 
Moreover, 
we can reduce the model size significantly by generating the constraints   (\ref{lin:unb1p3}--\ref{lin:uconflict}) 
through lazy callback by inspecting potentially feasible solutions. 
In other words,
we inspect the covering arrays associated with solutions of the IP formulation without those constraints, 
and then dynamically add the constraints that would make those solutions infeasible if the covering array is not valid.

\subsection{Conjectured Exact Approach $\Mone$}\label{sec:mip1}

If we assume Conjecture~\ref{conj:balanced}, 
then we use only balanced columns for every unconstrained parameter $X_i, i \in \tilde{\mathcal{P}}$, 
which reduces the number of choices considerably.

One simple way to produce that modification would be to restrict the compatibility of unconstrained parameters to only canonical balanced columns in the preprocessing, 
hence not modifying the previous formulation at all.

\subsection{Baseline Exact Approach $\Mb$}\label{sec:big_mip}

Now we present a baseline model in which we model the assignments in the covering array as decision variables. 
The results previously discussed are not necessary for this model, except for Assumption~\ref{ass:last}: 
we state that every pair of assignments that is not forbidden should be covered in at least one row, 
which thus renders the model infeasible if an implicit pairwise constraint exists. 

In this model, we use the additional index and elements:
\begin{itemize}
\item $r$ for the number of the row in the covering array; 
\item the set of rows $\mathcal{R} := \{ 1, \ldots, N^{\max}\}$, where $N^{\max}$ is chosen in advance.
\end{itemize}

\ifpreprint
    % Nada!
\else
    \MbaselineMILP
\fi

In the IP model~\eqref{eq:big_obj}--\eqref{eq:big_last}, 
we represent the covering array assignment directly with the binary variable $M_{r i}$ denoting the value of parameter $i$ (column $i$) in test $r$ (row $r$).
In addition, 
we use the binary variable $Y_{r i_1 i_2 \alpha \beta}$ to denote if test $r$ covers the pairwise assignment $(X_{i_1}, X_{i_2}) = (\alpha, \beta)$; 
and the binary variable $Z_r$ to denote if row $r$ is used in the covering array. 
If a row is not used, we skip that row when producing the covering array from the solutions obtained.

We relate the $Y$ variables with a value of one to the corresponding assignment on the $M$ variables with~\eqref{eq:big_ym1}--\eqref{eq:big_ym2}. 
We ensure that at least one row covers each allowed pairwise assignment on the $Y$ variables with~\eqref{eq:big_y11} when there is no pairwise constraint between the parameters, 
and with~\eqref{eq:big_y12} when there is a pairwise assignment between the parameters. 
We ensure that no row covers a forbidden pairwise assignment with~\eqref{eq:big_y0} on the $Y$ variables; 
and with~\eqref{eq:big_mm1}--\eqref{eq:big_mm2} on the $M$ variables.
Finally, 
we limit the value of one on the $Y$ variables when the value on the corresponding $Z$ variable is one with~\eqref{eq:big_yz}.

The objective function~\eqref{eq:big_obj} then induces an optimal solution corresponding to an optimal covering array, 
provided that we choose a sufficiently large value for $N^{\max}$ in advance. 
However, using a much larger for $N^{\max}$ than needed may result in a much larger model, 
which may not be solved as fast. 

This model has many forms of symmetry. 
However, 
we could not obtain any significant improvement in preliminary experiments by exploiting some of those in the unconstrained case. 
Nevertheless, 
we believe that this model is an earnest alternative if we were to ignore the developments proposed in this paper, 
given that the authors formulated and relied on it as part of their initial investigation.

\ifpreprint
    \MbaselineMILP
\fi

\section{Computational Experiments}\label{sec:results}

We designed our computational experiments to evaluate the four approaches described: 
the exact method $\Mzero$, the conjectured exact method $\Mone$, the baseline exact method $\Mb$, and the heuristic method $\text{H}$.  
We also compare some of those approaches with commonly used heuristics from the literature by using two solvers, 
the Automated Combinatorial Testing for Software (ACTS) 3.2 solver \cite{NIST_Combinatorial_Testing_Tools} and the Constrained Covering Array Generation (CCAG) solver~\cite{CCAG_GIST_NJU}. 
ACTS was chosen because of its historical importance and continuous development.
CCAG was chosen because it is a framework implementing and comparing the most popular heuristics in its accompanying survey~\cite{wu2021comparative}.   
We refer to ACTS as one of the heuristics tested. 
We also test heuristics AETG, DDA, PSO, SA, and TS from CCAG along with their best-performed handlers identified in~\cite{wu2021comparative}. 

\subsection{Implementation Details}

We generated 300 random binary instances by combining $|\mathcal{P}| \in \{10,20,30,40,50\}$ parameters with $|\mathcal{F}| \in \{1,5,10,20,40,80\}$ forbidden pairs, using 10 random seeds for each setting.
All methods were given a 30-minute limit per instance, except for ACTS because it is preconfigured and runs to completion. 
For the CCAG heuristics, 
we record the best solution obtained within this budget and the first time in which such a solution was obtained, 
and the number of iterations is chosen to exceed the time limit and the solver is interrupted at the time limit. 
For heuristic H,
we used the following settings:  
$C_1 = 6$, 
$C_2 = 4$, 
$P = 9$, 
$T_N = 180$, 
$R_1 = 8$, 
$R_2 = 8$, and
$T_F = 60$.

All experiments were run on identical Red Hat Enterprise Linux~9.6 nodes with Intel Xeon Gold 6226 (Cascade Lake) processors at 2.70~GHz, 
with each run confined to a single core and 32~GB of RAM. The methods were implemented in Python~3.9.21 and MILP models were solved with the Gurobi~12.0.1 solver.

\subsection{Results and Analysis}

We collected data and visualized the results with a few questions in mind. 
First, 
we evaluated the exact approaches to answer the following questions: 
\begin{enumerate}[(i)]
    \item \emph{How does the proposed approach $\Mzero$ compare with the baseline $\Mb$?}
    \item \emph{Is there a refutation of Conjecture~\ref{conj:balanced} in the results with approach $\Mone$?}
    \item \emph{If not and assuming Conjecture~\ref{conj:balanced}, how does $\Mone$ compare with $\Mzero$?}
\end{enumerate}

\ifpreprint
    \FigPerformance
\fi

Figure~\ref{fig:performance} shows a cumulative plot of the proportion of instances solved by each exact method, 
including both when an optimal solution is found and when the instance is proven infeasible, 
in logarithmic and linear time scale. 
We note that the visible uptick of instances solved within a hundredth of a second by both $\Mzero$ and $\Mone$ 
correspond to cases in which preprocessing has proven infeasibility. 
In total, 76 of the 300 instances are proven infeasible by preprocessing. 
That includes all the instances with 80 pairwise constraints. 

\ifpreprint
    %Nada
\else
    \FigPerformance
\fi

Figure~\ref{fig:exact-rows} compares the solution value (number of rows) of the covering arrays obtained by exact methods for each instance. 
In this and other scatter plots, 
we exclude infeasible instances, represent timeout results with a large value labeled as T.O., 
and use larger markers with a number inside when multiple instances produce the same pair of values. 
If an optimal solution is found by both methods, we expected---and indeed observed---that the solution values match. 
Hence, 
the only differences are in the timeouts summarized at the top left of each plot. 

\FigExactRows

Figure~\ref{fig:exact-times} compares the time taken by exact methods to solve each instance, 
allowing us to evaluate which method solved each instance faster. 
Those cases are summarized at the bottom right of each plot. 

\begin{figure}[h!]
\centering
\includegraphics[width=\textwidth]{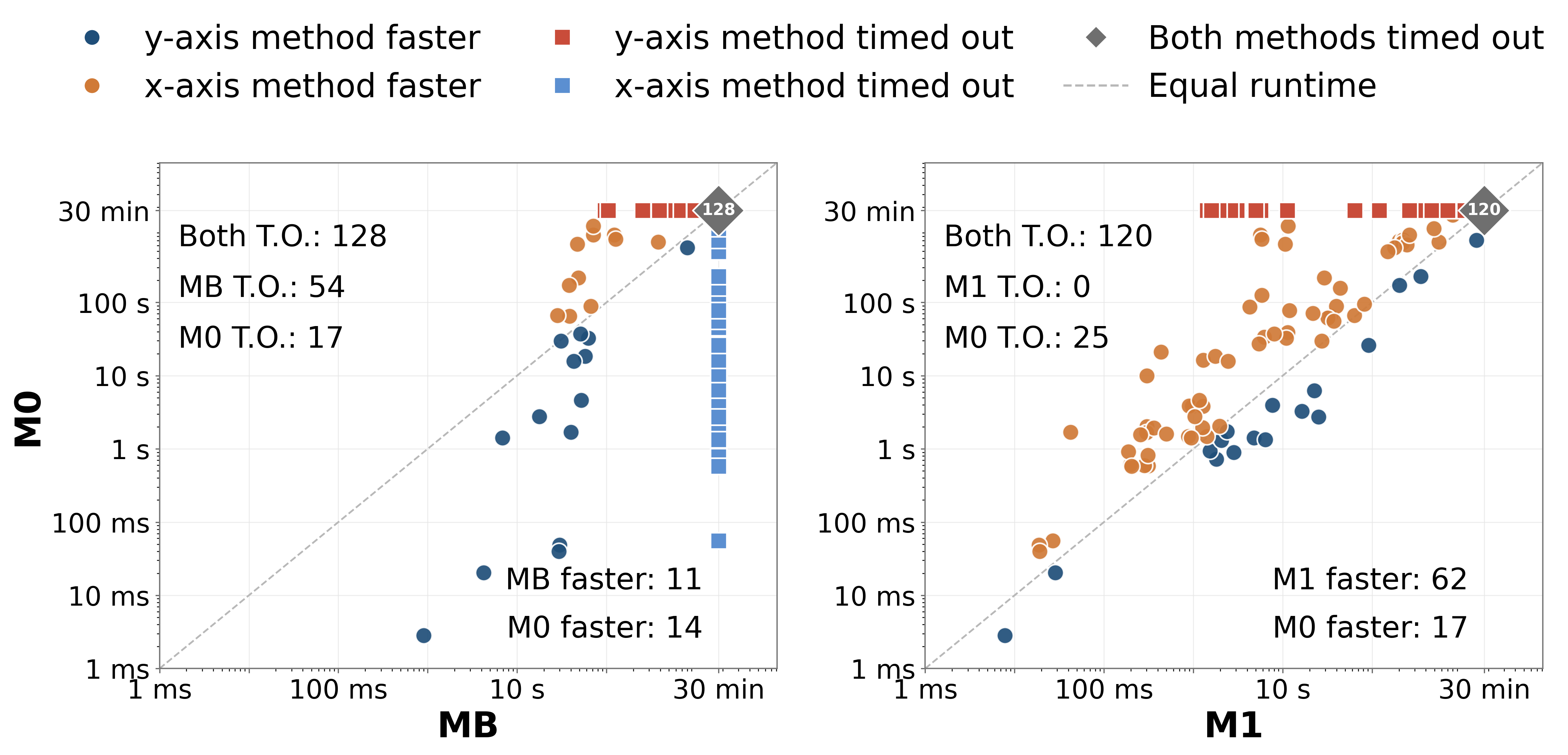}
\caption{Pairwise comparison of runtimes and timeouts for exact methods, 
contrasting $\Mb$ with $\Mzero$ (left) and $\Mone$ with $\Mzero$ (right).}
\label{fig:exact-times}
\end{figure}

\paragraph{Answering (i)} 
We observe that $\Mzero$ has solved more instances than $\Mb$ at any given time (Figure~\ref{fig:performance}), 
although both methods solved instances that the other did not (Figure~\ref{fig:exact-rows} left). 
Among those, 
we observe instances solved up to a minute by $\Mb$ being solved faster by $\Mzero$; whereas 
the instances in which $\Mzero$ times out take at least 100 seconds to be solved by $\Mb$, 
while most of the instances in which $\Mb$ times out take at least 1 second to be solved by $\Mzero$ (Figure~\ref{fig:exact-times} left). 
Hence, 
$\Mzero$ generally solves more instances and faster than $\Mb$. 

\paragraph{Answering (ii)} 
There is no instance solved by both $\Mzero$ and $\Mone$ in which the optimal solution value is not the same (Figure~\ref{fig:exact-rows} right). 
Hence, 
we have not observed a contradiction with Conjecture~\ref{conj:balanced}, 
which remains unrefuted. 

\paragraph{Answering (iii)}  
We observe that $\Mone$ similarly has solved more instances than $\Mzero$ at any given time (Figure~\ref{fig:performance}), 
and in fact the instances solved by $\Mzero$ are a proper subset of those solved by $\Mone$ (Figure~\ref{fig:exact-rows} right). 
Among the instances solved by both, 
the majority was solved faster by $\Mone$, sometimes with orders-of-magnitude speedups, 
whereas the instances solved faster by $\Mzero$ were within about an order-of-magnitude speedup (Figure~\ref{fig:exact-times} right). 
Hence, 
$\Mone$ performs better than $\Mzero$, 
solving the same and even more instances, and generally faster.

Second, 
we evaluated our heuristic approach $\text{H}$ with the following questions:
\begin{enumerate}[(i)]
    \setcounter{enumi}{3}
    \item \emph{How does $\text{H}$ compare with the exact approaches $\Mzero$ and $\Mone$?}
    \item \emph{How does $\text{H}$ compare with with other commonly used heuristics?}
    \item \emph{Is there a refutation of Conjecture~\ref{conj:balanced} in the results with any heuristic?}
\end{enumerate} 

We use similar scatter plots as before 
to compare $\text{H}$ with the exact methods $\Mzero$ and $\Mone$: 
Figure~\ref{fig:heuristic-rows} compares the solution value (number of rows) of the covering arrays obtained and 
Figure~\ref{fig:heuristic-times} compares the runtimes involved. 

\begin{figure}[h!]
\centering
\includegraphics[width=\textwidth]{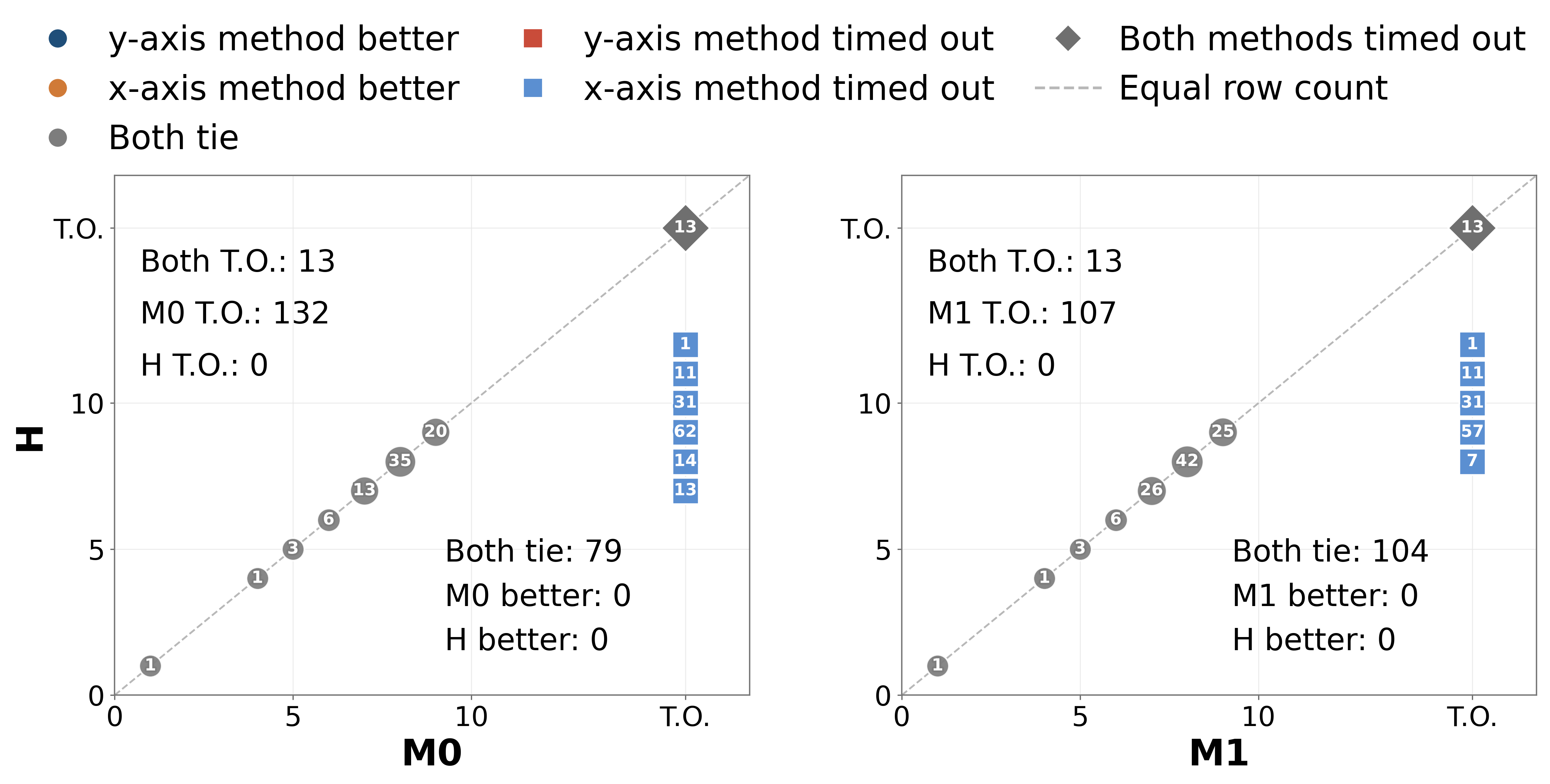}
\caption{Pairwise comparison of solution values (numbers of rows) and timeouts, 
contrasting exact methods $\Mzero$ (left) and $\Mone$ (right) with heuristic $\text{H}$.}
\label{fig:heuristic-rows}
\end{figure}

\begin{figure}[h!]
\centering
\includegraphics[width=\textwidth]{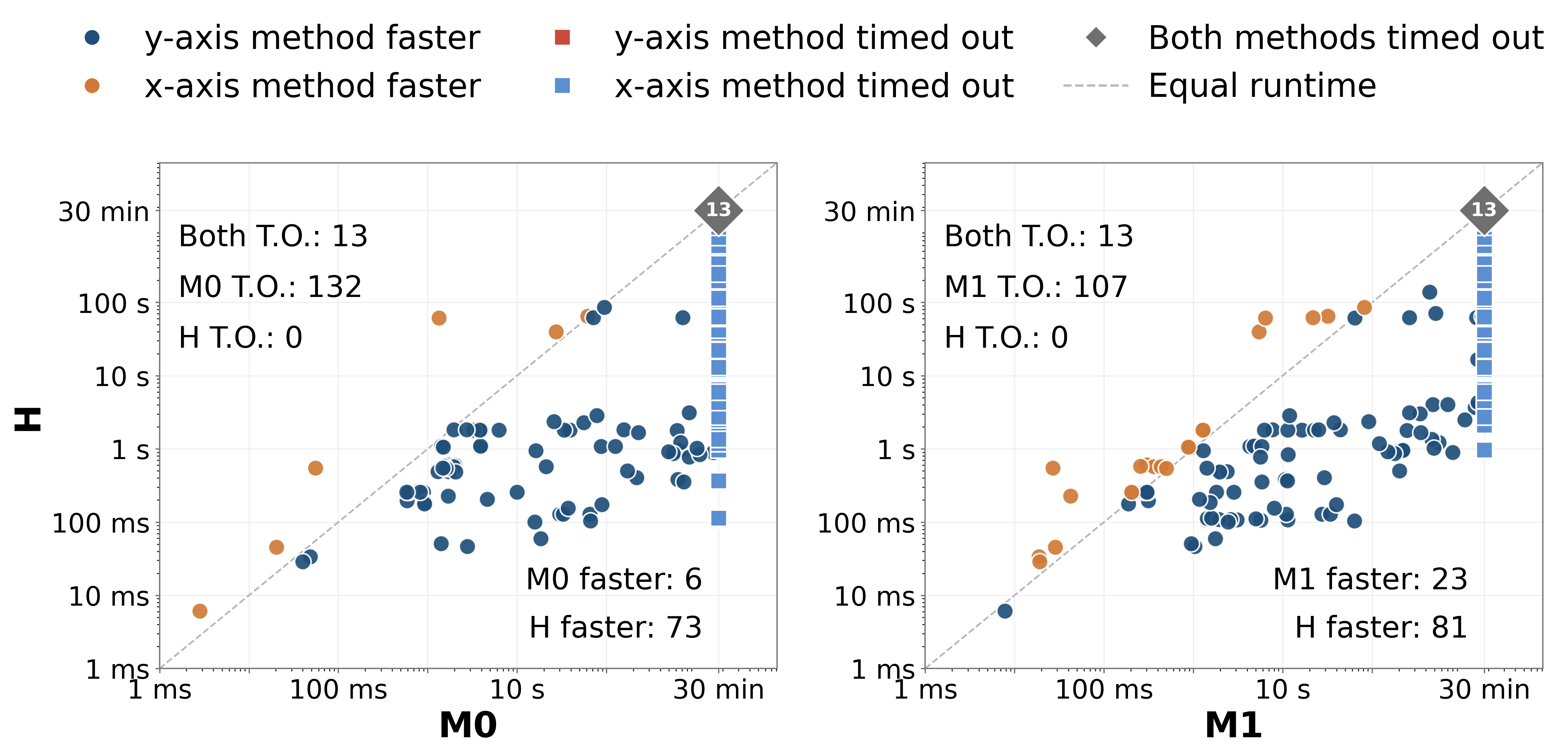}
\caption{Pairwise comparison of runtimes and timeouts, 
contrasting exact methods $\Mzero$ (left) and $\Mone$ (right) with heuristic $\text{H}$.}
\label{fig:heuristic-times}
\end{figure}

Figure~\ref{fig:heuristicS-rows} compares the solution value (number of rows) of the covering arrays obtained by heuristic methods, 
comparing $\text{H}$ with each other method. 
We conclude from these plots that $\text{SA}$ and $\text{TS}$ are the best performing heuristics besides $\text{H}$, 
for which reason we restrict the subsequent analysis to only those. 

In the case of ACTS, 
we note that a covering array is also produced for the 76 instances that were proven to be infeasible in our preprocessing. 
Since infeasible solutions is not included in the scatter plots, that did not affect the comparisons. 
In the case of the heuristics within CCAG (AETG, DDA, PSO, SA, and TS), 
the plots would have indicated 76 timeouts if those instances were included. 

\begin{figure}[h!]
\centering
\includegraphics[width=\textwidth]{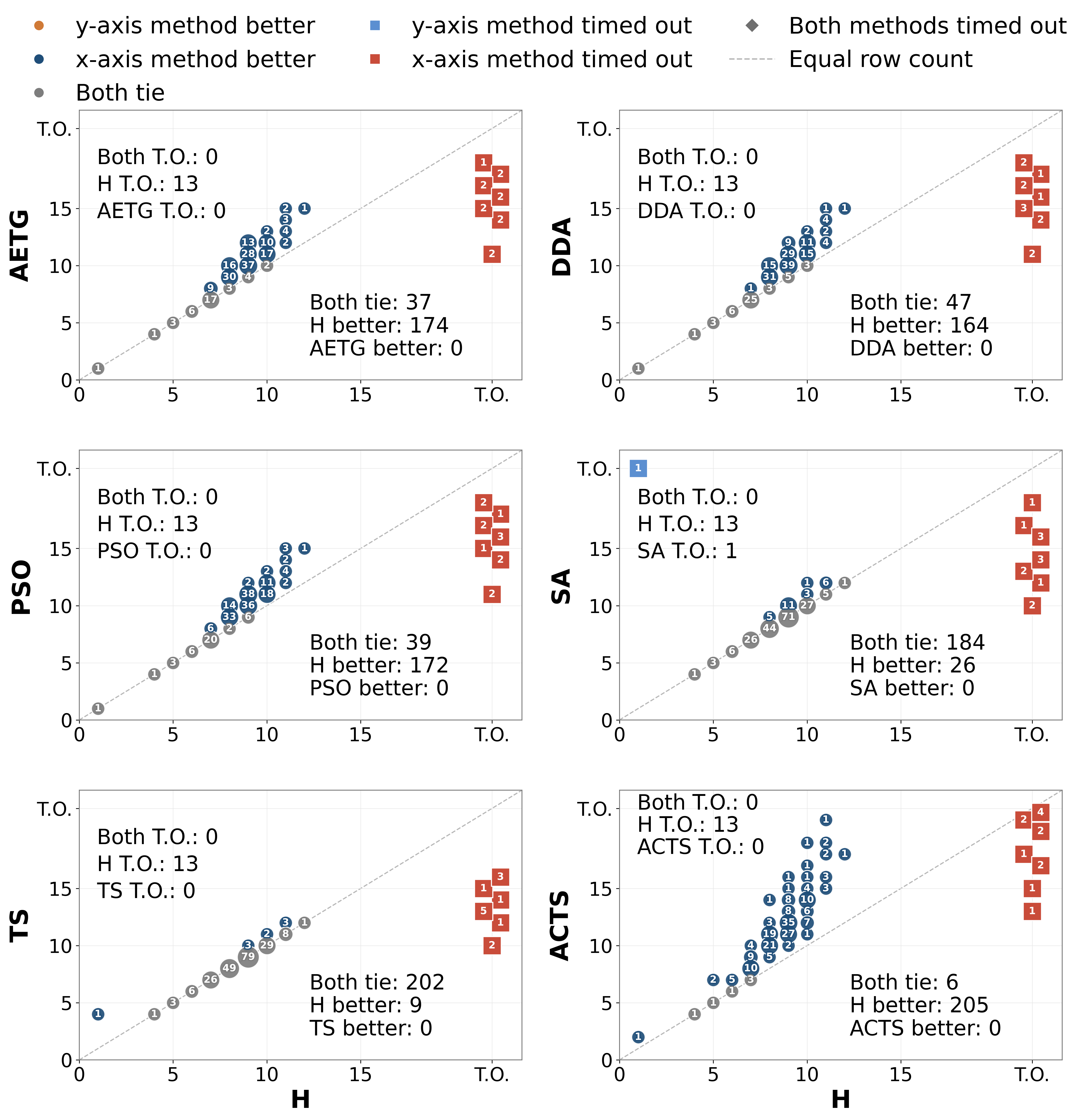}
\caption{Pairwise comparison of solution values (numbers of rows) and timeouts, 
contrasting heuristic $\text{H}$ with heuristics AETG, DDA, PSO, SA, TS, and ACTS.}
\label{fig:heuristicS-rows}
\end{figure}

Figure~\ref{fig:heuristicS-times} compares the time taken by $\text{H}$ to solve each instance 
with the time that it took SA and TS to obtain their best solution value for the first time. 

\begin{figure}[h!]
\centering
\includegraphics[width=\textwidth]{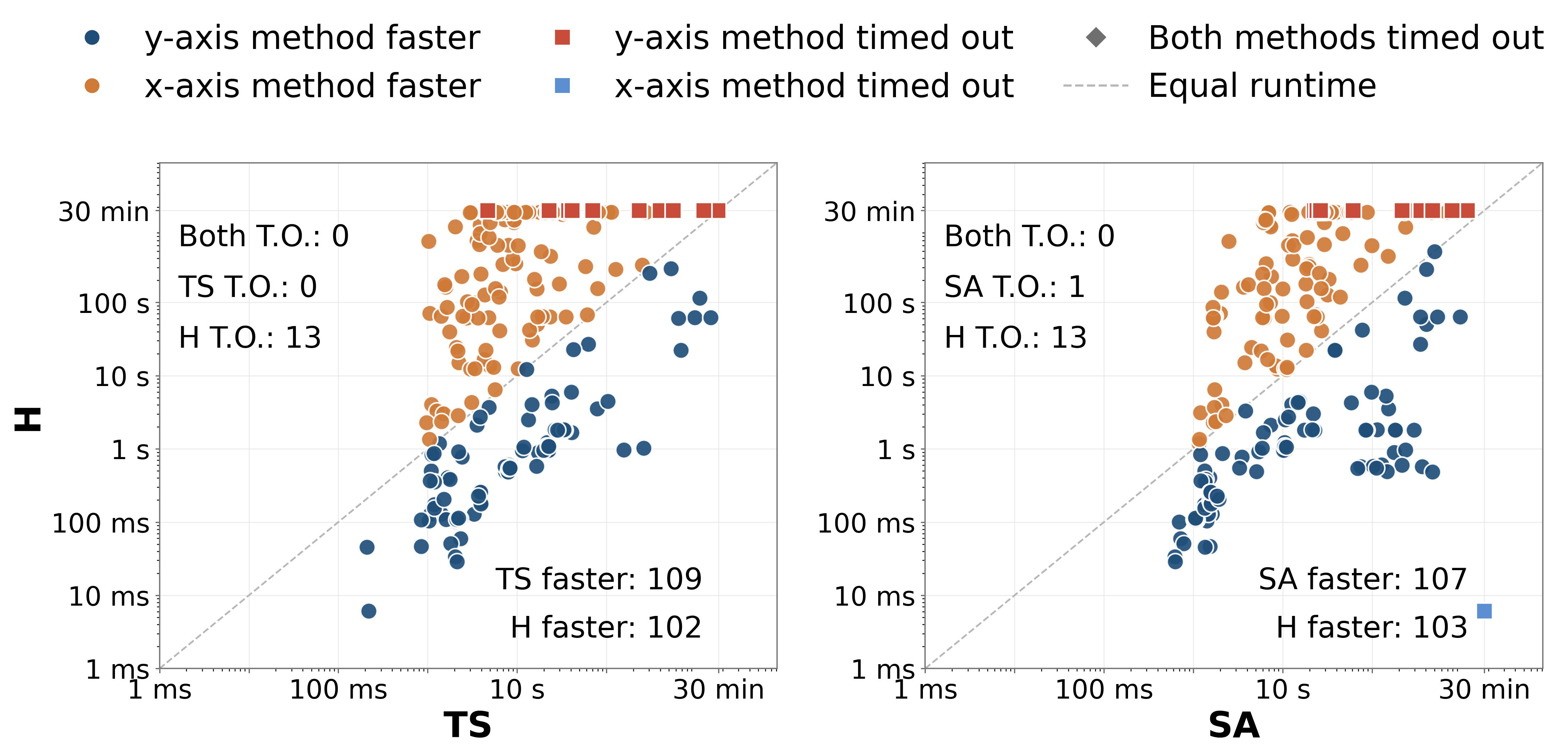}
\caption{
Pairwise comparison of time for finding the best solution and timeouts, 
contrasting heuristic $\text{H}$ with the best performing heuristics SA and TS.}
\label{fig:heuristicS-times}
\end{figure}

\paragraph{Answering (iv)} 
We note that $\text{H}$ matches the solution value (number of rows) of all the solutions obtained with $\Mzero$ and $\Mone$, 
more than doubling the total number of instances solved with respect to each (Figure~\ref{fig:heuristic-rows}).  
Among the instances that each exact method also solved, 
a majority is solved at least one order-of-magnitude faster by $\text{H}$ (Figure~\ref{fig:heuristic-times}). 
Hence, 
$\text{H}$ has solved twice as many instances as $\Mzero$ and $\Mone$,  
and the solutions obtained for instances also solved by $\Mzero$ or $\Mone$ were generally solved faster and all happened to be optimal---although $\text{H}$ is an heuristic and thus bound to occasionally produce suboptimal solutions. 

\paragraph{Answering (v)} 
We note that $\text{H}$ produced solutions with as many or fewer rows than the other heuristics; 
although it occasionally timed out without a solution, 
and the same happened but once with another heuristic if we ignore the 76 timeouts due to infeasibility proven by preprocessing (Figure~\ref{fig:heuristicS-rows}). 
Compared with the best-performing competing heuristics SA and TS in the cases in which the instances are solved by $\text{H}$ and each of those, 
$\text{H}$ solves almost half of the instances faster in each case (Figure~\ref{fig:heuristicS-times}). 
Hence, 
balancing time outs and solution values in feasible instances, 
$\text{H}$ is competitive with SA and TS, 
narrowly better handling more instances than SA (26 better solutions vs. 13 against 1 timeouts) and fewer instances than TS (9 better solutions vs. 13 timeouts); 
whereas SA and TS timeout in the 76 instances that $\text{H}$ identifies as infeasible. 

\paragraph{Answering (vi)}
There are no instances solved by $\Mone$ for which a better solution is found by heuristic $\text{H}$ (Figure~\ref{fig:heuristic-rows}). 
Since $\text{H}$ solved all the instances also solved by $\Mone$ (Figure~\ref{fig:heuristic-rows}) and no other heuristic produced a better solution than $\text{H}$ in the cases in which $H$ produced a solution (Figure~\ref{fig:heuristicS-rows}), 
then by consequence the previous statement applies to the other heuristics tested. 
Hence, 
we continue not having observed a contradiction with Conjecture~\ref{conj:balanced}, 
which remains unrefuted.

Finally,  
we study in more detail the instances used with respect to number of parameters and pairwise constraints to answer the following questions: 
\begin{enumerate}[(i)]
    \setcounter{enumi}{6}
    \item \emph{How often each type of instance is solved to optimality or proven infeasible?}
    \item \emph{How does the optimal solutions compare with the unconstrained case?}
\end{enumerate} 

We use two new types of plots in this part. 
Figure~\ref{fig:m1-status} distributes the instances with respect to the solutions obtained with $\Mone$ among optimal, timed out, and proved infeasible within each number of parameters and pairwise constraints. 
Within each such combination of number of parameters and pairwise constraints, 
Figure~\ref{fig:m1-vs-unconstrained} distributes the instances with respect to the value of the optimal solution (number of rows) in comparison to the unconstrained case. 

\begin{figure}[h!]
\centering
\includegraphics[width=\textwidth]{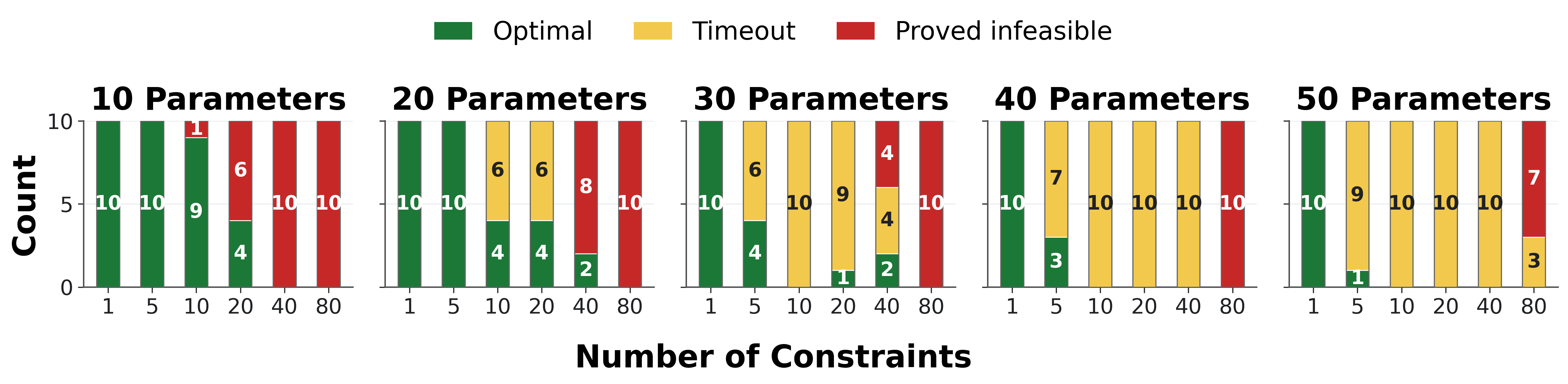}
\caption{Instances solved to optimality, timed out, and proven infeasible by $\Mone$ for each number of parameters and pairwise constraints.}
\label{fig:m1-status}
\end{figure}

\begin{figure}[h!]
\centering
\includegraphics[width=\textwidth]{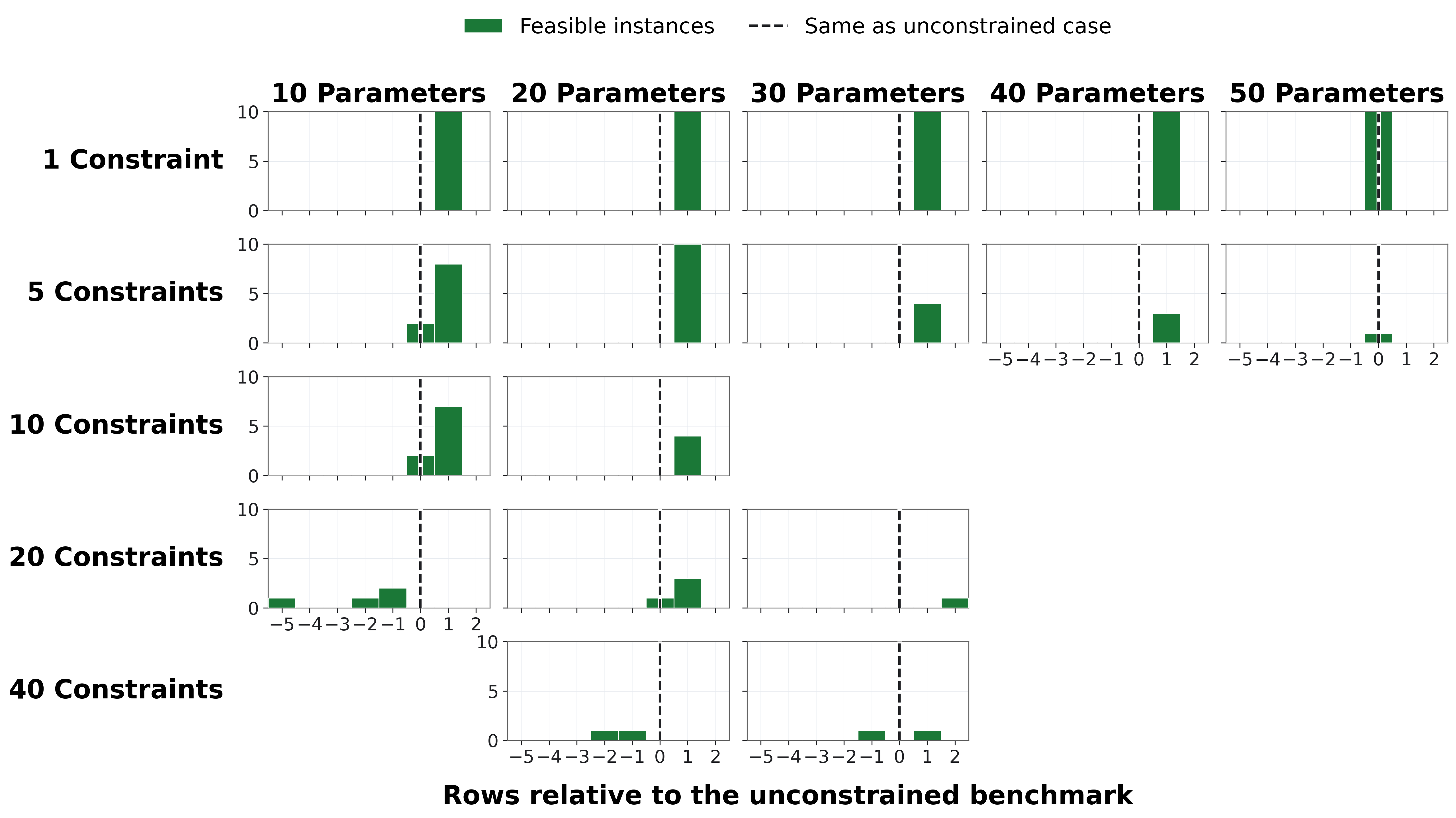}
\caption{Difference in optimal solution value (row count) with respect to the unconstrained case for instances solved to optimality by $\Mone$ per each number of parameters and pairwise constraints. Empty cells have no feasible instances.}
\label{fig:m1-vs-unconstrained}
\end{figure}

\paragraph{Answering (vii)} 
Increasing the number of pairwise constraints shifts the most common outcome from optimal to proved infeasible, 
whereas increasing the number of parameters creates and widens a gap between those cases in which the most common outcome is timeout (Figure~\ref{fig:m1-status}).

\paragraph{Answering (viii)} 
Adding from one to ten random constraints generally increases the number of rows in the optimal solutions when compared to the unconstrained case;  
with twenty constraints or more, we observe both increases and decreases (Figure~\ref{fig:m1-vs-unconstrained}). 
Although we have observed two pairwise constraints being enough for the optimal covering array to have fewer rows (see discussion about $\BFM^{(8)}$ at the beginning of Section~\ref{sec:pairwise}), 
we note that those constraints were not random.

\section{Conclusion}

In this paper, 
we have studied how to solve the coverage problem in combinatorial testing 
for binary parameters and subject to pairwise constraints. 
First, we generalized results on the structure of optimal solutions by showing that: 
\begin{enumerate}[(i)]
\item in the absence of constraints, the constructive algorithm widely used in the literature may in some cases generate all optimal solutions; 
whereas 
\item constrained testing may lead to comparatively more tests, but also occasionally to paradoxically fewer tests---even among binary parameters; 
\item constraints on subsets that should otherwise be covered fundamentally change the structure of testing suites; but, perhaps more importantly, 
\item the key to finding optimal solutions remains to be on balancing the number of zeros and ones per parameter, as much as still possible.
\end{enumerate}
In fact, the prevalence of balanced columns has been empirically observed in unconstrained optimal solutions~\cite{ANewBacktrackingAlgorithmforConstructingBinaryCoveringArraysofVariableStrength}; 
and in a sense relates to the alternative---but less flexible---use of orthogonal arrays~\cite{Williams2000}, 
in which all pairs of assignments must appear the same number of times.

Second,  
we showed (a) how to obtain optimal testing suites by solving a series of Integer Linear Programming~(ILP) models; 
and (b) how to obtain similar solutions faster with an heuristic that is competitive with commonly used techniques in quality and runtime. 
To the best of our knowledge, 
this is the first non-enumerative approach to obtain the optimal solution in the setting of binary parameters subject to pairwise constraints.

In future work, 
we hope to develop new theoretical results on the minimum number of tests and on the feasibility conditions for covering arrays, 
similar to those in Sections~\ref{sec:lb} and \ref{sec:cliques}, 
and explore more efficient approaches to solve the covering problem to optimality.

\begin{credits}
\subsubsection{\ackname} Changkun Guan was supported by Bucknell University's PP\&L Undergraduate Research Fund. Hunter Gehman was supported by Bucknell University's Freeman College of Management Gift Funds. Hunter Gehman and Mikey Ferguson were supported by Bucknell University's Presidential Fellowship. 
\end{credits}

 \bibliographystyle{splncs04}
 \bibliography{references}

\end{document}